\documentclass[11pt]{amsart}

\usepackage{amsmath, amsthm, amssymb}

\pdfoutput=1
\usepackage{palatino} 
\usepackage[abs]{overpic}                 

\usepackage{xcolor}
\definecolor{indigo}{rgb}{0.29, 0.0, 0.51}
\usepackage[colorlinks, urlcolor=indigo, linkcolor=indigo, citecolor=indigo]{hyperref}

\makeatletter
\newcommand*\bigcdot{\mathpalette\bigcdot@{0.6}}
\newcommand*\bigcdot@[2]{\mathbin{\vcenter{\hbox{\scalebox{#2}{$\m@th#1\bullet$}}}}}
\makeatother

\newtheorem{question}{Question}
\newtheorem{theorem}{Theorem}

\newtheorem{lemma}[theorem]{Lemma}

\theoremstyle{definition}

\theoremstyle{remark}
\newtheorem{remark}[theorem]{Remark}
\newtheorem{example}[theorem]{Example}

\def\dfn#1{{\em #1}}
\def\R{\mathbb{R}}
\def\Z{\mathbb{Z}}

\DeclareMathOperator\tb{tb}
\DeclareMathOperator\rot{r}
\DeclareMathOperator\tw{tw}
\DeclareMathOperator\under{u}
\def\tbb{\overline {\tb}}

\numberwithin{theorem}{section}
\theoremstyle{plain}

\begin{document}

\title{Cabling Legendrian and transverse knots}

\author{Apratim Chakraborty}

\author{John B. Etnyre}

\author{Hyunki Min}

\address{TCG CREST, Kolkata, India}
\email{apratimn@gmail.com}

\address{School of Mathematics \\ Georgia Institute
of Technology \\  Atlanta, Georgia}
\email{etnyre@math.gatech.edu}

\address{Department of Mathematics \\ Massachusetts Institute of Technology \\  Cambridge, Massachusetts}
\email{hkmin@mit.edu}


\begin{abstract}
In this paper we will show how to classify Legendrian and transverse knots in the knot type of ``sufficiently positive" cables of a knot in terms of the classification of the underlying knot. We will also completely explain the phenomena of ``Legendrian large" cables. These are Legendrian representatives of cables that have Thurston-Bennequin invariant larger than the framing coming from the cabling torus. Such examples have only recently, and unexpectedly, been found. We will also give criteria that determines the classification of Legendrian and transverse knots the knot type of negative cables. 
\end{abstract}

\maketitle

\section{Introduction}

There have been many partial results concerning the classification of Legendrian knots in cabled knot types, and the study of such Legendrian knots has greatly enhanced our understanding of the behavior of Legendrian knots. For example, the first classification of Legendrian and transversely non-simple knot types was given in \cite{EH04} and even more exotic phenomena was observed in \cite{ELT12}.  In \cite{EH04}, the second author and Honda classified Legendrian and transverse cables when the underlying knot type is Legendrian simple and uniformly thick and also showed the $(3,2)$-cable of the right handed trefoil was not Legendrian simple. In \cite{Tosun13}, Tosun obtained further classification results for positive cables when the underlying knot type is Legendrian simple and the contact width is an integer. In \cite{ELT12}, the second author, LaFountain and Tosun completely classified Legendrian and transverse cables of torus knots. A general approach to studying Legendrian representatives of satellite knots was explored by the second author and V\'ertesi in \cite{EtnyreVertesi18}. 

In this paper, we will show how to completely classify the Legendrian and transverse knots in a sufficiently positive cable of a knot type $K$ in terms of the  classification of Legendrian and transverse knots in the knot type $K$. We also completely explain the phenomena of ``Legendrian large cables" \cite{Mccullough18}, that are Legendrian representatives of $(p,q)$-cables with Thurston-Bennequin invariant larger than $pq$, which prior to the recent work of Yasui \cite{Yasui16}, was thought to be the upper bound on such cables. Finally we describe criteria that allows one to understand negative cables of knots. 

\smallskip
\noindent
{\bf Slope and cabling conventions.} Given an oriented null-homologous knot $K$ in a $3$--manifold $Y$, it has a solid torus neighborhood $N$. We can take $\lambda$ to be the curve on $\partial N$ that bounds a Seifert surface in $\overline{Y-N}$ and $\mu$ the curve on $\partial N$ that bounds a disk in $N$. For $p$ and $q$ relatively prime, the $(p,q)$-cable of $K$ will be the curve on $\partial N$ in the homology class $p[\lambda]+q[\mu]$. We will always take $p>0$ since one can consider cables with negative $p$ as cables of $-K$. We will also call this curve a slope $q/p$ curve on $\partial N$. {\em Warning:} this slope convention is different from the one used in many of the early papers in contact geometry. In those papers this slope would be called $p/q$. We adopt this convention as it agrees with the standard convention used when describing cables and surgery in topology.

\subsection{Cabling and simplicity of knots}\label{introgrc}

Recall a knot type $K$ is called \dfn{Legendrian simple} if two Legendrian knots in the knot type are Legendrian isotopic if and only if they share the same Thurston-Bennequin invariants and rotation numbers. Similarly $K$ is \dfn{transversely simple} if two transverse knots in the knot type are transversely isotopic if and only if they share the same self-linking numbers. 

We begin with some corollaries of our main results that roughly show that 
\begin{enumerate}
\item for sufficiently positive cables of a knot type, the Legendrian and transverse classification is as complex as it is for the underlying knot,
\item for sufficiently negative cables, the classification can become simpler, and
\item for cables of slope near the maximal Thurston-Bennequin invariant of a knot, the classification can become more complex. 
\end{enumerate}
The first statement is made precise in the following results; but first recall that the contact width of a knot type $K$ is $\omega(K)$, the supremum of the dividing slopes of all convex tori which bound a solid torus representing the knot type $K$. 

\begin{theorem}\label{tm1}
  A knot type $K$ in $(S^3,\xi_{std})$ is Legendrian simple if and only if the cabled knot type $K_{(p,q)}$ is Legendrian simple for any $q/p>\lceil \omega(K) \rceil$. 

  A knot type $K$ in $(S^3,\xi_{std})$ is transversely simple if and only if the cabled knot type $K_{(p,q)}$ is transversely simple for any $q/p>\lceil \omega(K) \rceil$. 
\end{theorem}

\begin{remark}
  Since $\tbb(K) \leq \omega(K) \leq \tbb(K)+1$ for any knot type $K$ in $(S^3,\xi_{std})$, Theorem~\ref{tm1} remains true if we replace $\lceil \omega(K) \rceil$ with $\tbb(K) + 1$. 
\end{remark}

Under an extra hypothesis we see the same result for transverse knots and sufficiently negative cables. The term ``sufficiently negative" will be made clear in Section~\ref{grc}.\\

A knot $K$ is said to have the uniform thickness property if any solid torus representing the knot type of $K$ can be thickened to a standard neighborhood of a Legendrian representative of $K$ and $\omega(K)$ is equal to the maximal Thurston-Bennequin number $\tbb(K)$ of Legendrian representatives of $K$.  
\begin{theorem}\label{tm2}
  Suppose $K$ is a uniform thick knot in $(S^3,\xi_{std})$. Then $K$ is transversely simple if and only if the cabled knot type $K_{(p,q)}$ is transversely simple for any $q/p$ that is sufficiently negative. 
\end{theorem}
However, one can have a Legendrian non-simple knot type whose sufficiently negative cable is Legendrian simple. We call a knot type $K$ \dfn{partially uniformly thick} if there is an integer $n$ such that any solid torus representing $K$ with convex boundary having dividing slope less than $n$ will thicken to a solid torus with convex boundary having two dividing curves of slope $n$. 
\begin{theorem}\label{tm3}
  Suppose $K$ is a partially uniform thick knot in $(S^3,\xi_{std})$. If $K$ is transversely simple, then for $q/p$ sufficiently negative  $K_{(p,q)}$ is Legendrian simple. 
\end{theorem}
\begin{example}
In the proof of Theorem~4.1 in \cite{EtnyreHonda03}, it is shown that many connected sums of negative torus knots are not Legendrian simple. However, in \cite{EH04} it was shown that negative torus knots (Theorem~1.2) and their connected sums (Theorem~1.4) are uniformly thick. Finally, Theorem~4.6 in \cite{EtnyreHonda03} says that the connected sum of negative torus knots is transversely simple. Thus the above theorem shows that while these connected sums are not Legendrian simple, sufficiently negative cables of them will be. 
\end{example}
\begin{example}
Item~(3) about cables above can be seen by considering $(p,q)$-cables of positive $(r,s)$-torus knots with $q/p\in (0,rs-r-s)$. It was shown in \cite{EH04} for one such cable of the trefoil and in \cite{ELT12} for many cables of any positive torus knot, that they are not Legendrian simple even though the underlying knot is \cite{EtnyreHonda01b}.
\end{example}

\subsection{Classification of positive cables}
Given a knot type $K$, let $\mathcal{L}(K)$ denote the Legendrian isotopy classes of Legendrian knots in the knot type $K$. 
The \dfn{mountain range of $K$} is 
\begin{enumerate}
  \item the image $M$ of the map
  \[
  \mathcal{L}(K)\to \Z\times \Z: L \mapsto (\rot(L),\tb(L)),
  \]
  \item for each lattice point $(a,b)\in M$, a list of $L$ mapping to that point, and
  \item arrows indicating where an element in the mountain range maps under positive and negative stabilization. 
\end{enumerate}

Given a point $(a,b)$ in the integer lattice $\Z^2$ we define the $(p,q)$ \dfn{diamond of $(a,b)$} to be the points
\begin{align*}
  D_{(p,q)}(a,b) = \{ (r,t) : \,\, &t + |r - pa| \leq pq- |pb-q|,\\ & t + |r- pa| \geq pq- |pb-q| -2p +2,\\ & t+r \text{ odd}  \}.
\end{align*}

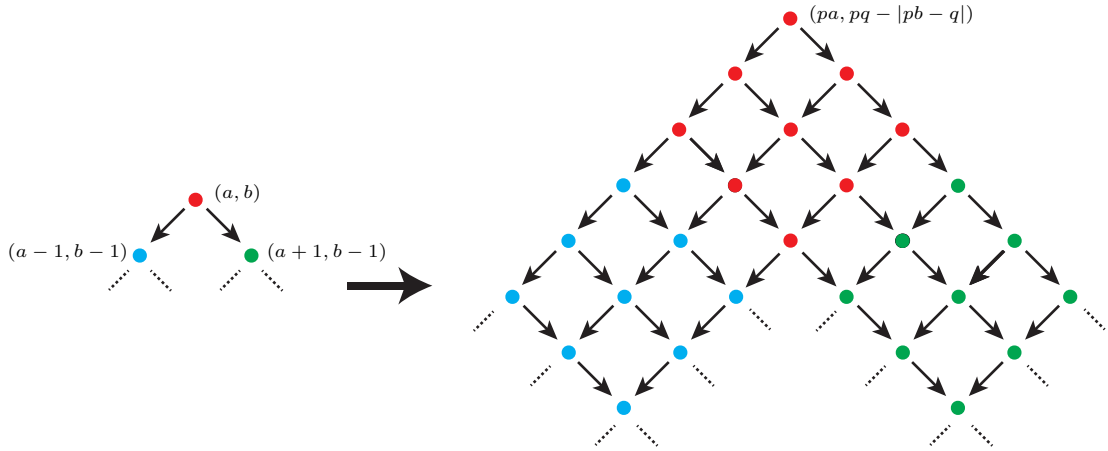
\begin{figure}[htb]{\tiny
\begin{overpic}
{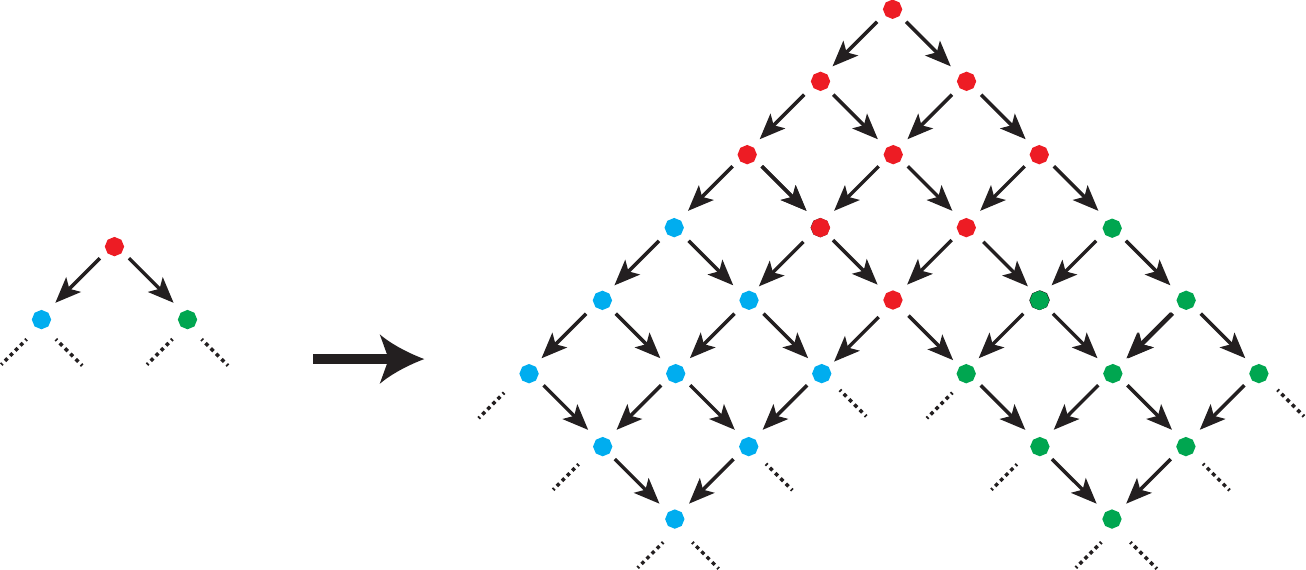}
\put(40, 94){$(a,b)$}
\put(60, 72){$(a+1,b-1)$}
\put(-38, 72){$(a-1,b-1)$}

\put(265, 162){$(pa,pq-|pb-q|)$}
\end{overpic}}
\caption{On the left are three lattice points in $\Z^2$. On the right is the $(p,q)$ diamond of these lattice points when $p=3$. The peaks of the diamonds are at $(pa,pq-|pb-q|)$, $(pa-p,pq-|pb-p-q|)$, and $(pa+p,pq-|pb-p-q|)$.}
\label{diamondex}
\end{figure}

See Figure~\ref{diamondex}. We may describe this diamond in a different way. We first describe stabilization of points $(a,b)\in \Z^2$ by $S_\pm(a,b)=(a\pm 1, b-1)$. Now $D_{(p,q)}(a,b)$ is the set of points obtained from the point $(pa, pq-|pb-q|)$ by between 0 and $p-1$ positive and 0 and $p-1$ negative stabilizations:
\[
D_{(p,q)}(a,b) = \{ S^k_+S^l_- (pa, pq-|pb-q|)) :  0\leq k,l, \leq p-1 \}.
\]
Given the points $M$ in a mountain range we define the \dfn{$(p,q)$ expansion of $M$} to be the set $M_{(p,q)}$ obtained from $M$ by replacing each point $(a,b)\in M$ with the diamond $D_{(p,q)}$:
\[
M_{(p,q)}= \bigcup_{(a,b)\in M} D_{(p,q)}(a,b).
\]
Notice that the diamonds in $M_{(p,q)}$ are disjoint. 

Given a Legendrian $L\in \mathcal{L}(K)$ we define its \dfn{$(p,q)$ diamond}, $D_{(p,q)}(L)$ as certain stabilizations of a ruling curve on its standard neighborhood. Specifically, given a Legendrian knot $L$ we will always denote by $L_{(p,q)}$ a ruling curve of slope $q/p$ on the boundary of a standard neighborhood of $L$ (see Section~\ref{kicm} for more on standard neighborhoods of Legendrian knots). With this notation we have that
\[
  D_{(p,q)}(L)= \{S^k_+S^l_- (L_{(p,q)}) :  0\leq k,l, \leq p-1\},
\]
where $S_\pm(L)$ is the positive/negative stabilization of $L$. 
Notice that the rotation numbers and Thurston-Bennequin invariants of the elements of $D_{(p,q)}(L)$ are precisely $D_{(p,q)}(\rot(L),\tb(L))$.

\begin{theorem} \label{classification-positive} 
Let $K$ be a knot in $(S^3,\xi_{std})$. If $q/p >  \lceil \omega(K) \rceil$ is not an integer, then 
  \[
    \mathcal{L}(K_{(p,q)}) = \bigcup_{L\in \mathcal{L}(K)} D_{(p,q)}(L).
  \]
  So each $L\in \mathcal{L}(K_{(p,q)})$ is associated to a unique knot $u(L)$ in $\mathcal{L}(K)$ such that $L\in D_{(p,q)}(u(L))$.  We call $u(L)$ the knot \dfn{underlying} $L$. Two knots in $ \mathcal{L}(K_{(p,q)})$ are Legendrian isotopic if and only if they have the same rotation numbers, Thurston-Bennequin invariants, and underlying knots. 

  In particular, the mountain range of $K_{(p,q)}$ is $M_{(p,q)}$ and any stabilizations of two distinct points in $M_{(p,q)}$ with the same invariants will stay distinct as long as the stabilization stays in the same $(p,q)$ diamond. If the $\pm$ stabilizations are in a distinct $(p,q)$ diamond, then they will be Legendrian isotopic if and only if their underlying knots $\pm$ stabilize to become Legendrian isotopic. 
\end{theorem}

\begin{figure}[htb]{
\begin{overpic}
{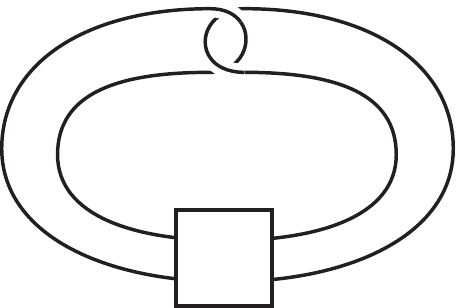}
\put(60, 11){$m$}
\end{overpic}}
\caption{The twist knot $K_m$, where in the box there are $m$ right handed half-twists.}
\label{twist-knots}
\end{figure}

The second author, Ng and V\'ertesi classified Legendrian twist knots in \cite{ENV13}. An immediate corollary to this and the above theorem is a classification of sufficiently positive cables (here, $q/p > \tbb(K)+1$) of twist knots. Such a result was inaccessible with previous work as some of the knot types are not Legendrian simple and not known to be uniformly thick. 

\begin{theorem}\label{classification-positive-twist}
  Let $K$ be the twist knot $K_m$, depicted in Figure~\ref{twist-knots}, with $m \neq -1$ half twists. 
  \begin{enumerate}
    \item If $m \geq -2$ even and $q/p> -m$, the knot type $K_{(p,q)}$ is Legendrian simple and there is a unique Legendrian knot in $\mathcal{L}(K_{(p,q)})$ with maximal Thurston Bennequin number $\tbb = pq-p(m+1)-q$ and rotation number $\rot=0$. See Figure~\ref{positive-cable-twist1}.
    \item If $m \geq 1$ odd and $q/p> -m-4$, the knot type $K_{(p,q)}$ is Legendrian simple and there are exactly two Legendrian knots in $\mathcal{L}(K_{(p,q)})$ with maximal Thurston Bennequin number $\tbb=pq-p(m+5)-q$ and rotation numbers $\rot=\pm p$. See  Figure~\ref{positive-cable-twist1}.
    \item If $m \leq -3$ odd and $q/p> -2$, the knot type $K_{(p,q)}$ is Legendrian non-simple and there are $-\frac{m+1}{2}$ Legendrian knots $L_i \in \mathcal{L}(K_{(p,q)})$, $i=1,...,-\frac{m+1}{2}$ with maximal Thurston Bennequin number $\tb = pq-3p-q$ and rotation number $\rot=0$. All other Legendrian knots destabilize to one of these knots and the $S_+^kS_-^l(L_i)$ are Legendrian isotopic if and only if  $k \geq p$ or $l \geq p$. See Figure~\ref{positive-cable-twist2}.
    \item If $m \leq -2$ even and $q/p > 2$, the knot type $K_{(p,q)}$ is Legendrian non-simple and there are $\lceil m^2/8 \rceil$ Legendrian knots $L_i \in \mathcal{L}(K_{(p,q)})$, $i=1,..., \lceil m^2/8 \rceil$ with maximal Thurston Bennequin number $\tb=pq-q+p$ and rotation number $\rot=0$. All other Legendrian knots destabilize to one of these knots. The $S_+^kS_-^l(L_i)$ fall into $\lceil -m/4 \rceil$ different Legendrian isotopy classes if $k \geq p, l < p$ or $k < p, l \geq p$, and $S_+^kS_-^l(L_i)$ all become Legendrian isotopic if $k \geq p$ and $l \geq p$. See Figure~\ref{positive-cable-twist2}. \hfill\qed
  \end{enumerate}
\end{theorem}

\begin{figure}[htb]{
\begin{overpic}
{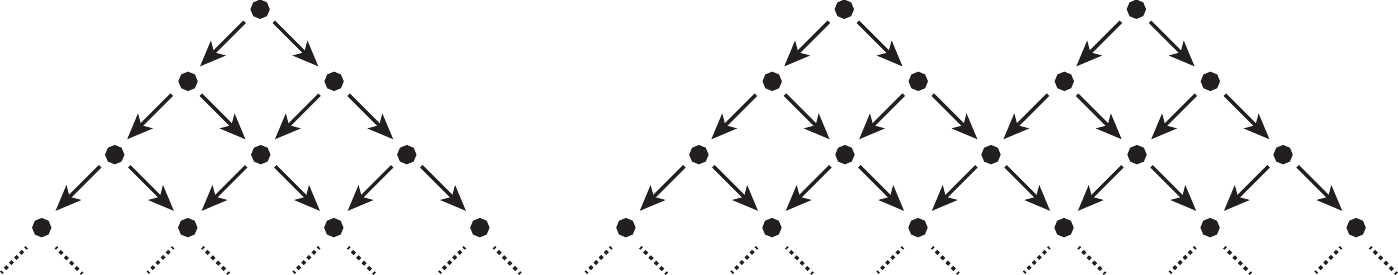}
\end{overpic}}
\caption{Mountain ranges for positive cables of the positive twist knots. On the left are $(p,q)$-cables of the positive twist knots with an even number $m$ of half twists and $q/p>-m$. On the right are $(p,q)$-cables of the positive twist knots with an odd number $m$ of half twists, $q/p>-m-4$ and $p=2$.}
\label{positive-cable-twist1}
\end{figure}

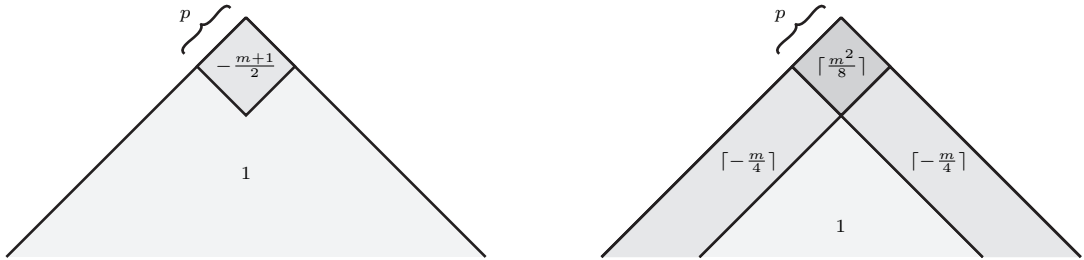
\begin{figure}[htb]{\tiny
\begin{overpic}
{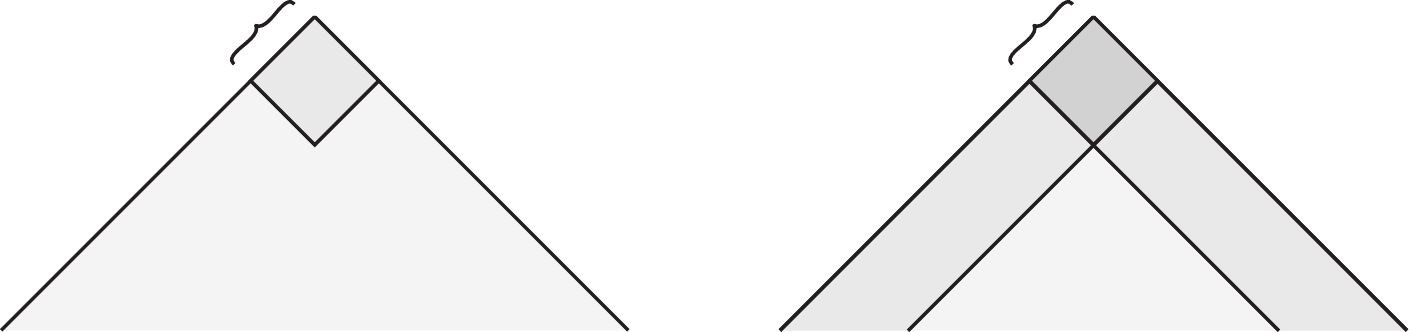}
\put(66, 91){$p$}
\put(291, 91){$p$}
\put(89, 30){$1$}
\put(79.5, 71){$-\frac{m+1}{2}$}
\put(306.5, 71){$\lceil \frac{m^2}{8} \rceil$}
\put(270, 35){$\lceil -\frac{m}{4} \rceil$}
\put(342, 35){$\lceil -\frac{m}{4} \rceil$}
\put(314, 10){$1$}
\end{overpic}}
\caption{Mountain ranges for positive cables of the negative twist knots. On the left are $(p,q)$-cables of the negative twist knots with an odd number $m$ of half twists and $q/p>-2$. On the right are $(p,q)$-cables of the negative twist knots with an even number $m$ of half twists and $q/p>2$.}
\label{positive-cable-twist2}
\end{figure}

Turning to transverse knots, let $\mathcal{T}(K)$ denote the transverse knots in the knot type $K$ up to transverse isotopy. For each $T\in \mathcal{T}(K)$ we define the \dfn{$(p,q)$ interval of $T$} as follows: choose a Legendrian approximation $L$ of $T$, and letting $I_{(p,q)}(T)$ be the set of transverse knots one obtains from the transverse push off of all the Legendrian knots in $D_{(p,q)}(L)$. We notice there will be exactly $p$ knots in $I_{(p,q)}(K)$ with self-linking numbers $p\, sl(T) + pq -q - 2k$ for $k=0,\ldots, p-1$.
\begin{theorem}\label{postransverse}
  If $q/p >  \lceil \omega(K) \rceil$ is not an integer, then
  \[
  \mathcal{T}(K_{(p,q)})= \bigcup_{T\in \mathcal{T}(K)} I_{(p,q)}(T).
  \] 
    So each $T\in \mathcal{T}(K_{(p,q)})$ is associated to a unique knot $u(T)$ in $\mathcal{T}(K)$ such that $T\in I_{(p,q)}(u(T))$. We call $u(T)$ the knot \dfn{underlying} $T$. Two knots in $ \mathcal{T}(K_{(p,q)})$ are transversely isotopic if and only if they have the same self-linking number and underlying knots. 
\end{theorem}

\subsection{Legendrian large cables}\label{introllc}
In \cite{Mccullough18}, McCullough defined a \dfn{Legendrian large cable} as a Legendrian knot $L \in \mathcal{L}(K_{(p,q)})$ with $\tb(L) > pq$. In \cite{Yasui16}, Yasui first introduced such knots. He showed that there is a $(n,-1)$-cable of the knots depicted in Figure~\ref{Yasui-knots} that has Thurston-Bennequin invariant $-1$. This was very surprising as the second author and Honda \cite{EH04} had shown for a uniform thick knot type $K$, one must have $\tbb(K_{(p,q)}) \leq pq$; and, moreover, after giving some evidence for the bound, Lidman and Sivek \cite{LidmanSivek16} conjectured the bound held for all cables.  The existence of Legendrian large cables is quite interesting since it implies there are Legendrian surgeries that produce reducible manifolds. See \cite{LidmanSivek16, Yasui16} for details. In \cite{Mccullough18}, McCullough also observed that any Legendrian large cable should be contained in a virtually overtwisted $T^2 \times I$. However, it was unclear which virtually overtwisted contact structures on $T^2 \times I$ contain Legendrian large cables. We will clarify this in the next theorem.
\begin{remark}
It is interesting to note that by Theorem~1.6 in \cite{LidmanSivek16} we know that $\mathcal{L}(K_{(p,q)})$ can contain Legendrian large cables only if $q/p<0$.
\end{remark}

A \dfn{length $2m$ balanced continued fraction block} is a length $2m$ continued fraction block with the same number of positive and negative basic slices. See Section~\ref{cfb} for more on continued fraction blocks. We say that the \dfn{center slope} of the balanced continued fraction block is the dividing slope of a convex torus between the $m^\text{th}$ and $(m+1)^\text{st}$ basic slices.

\begin{theorem}\label{llc-model} 
Let $(Y,\xi)$ be a tight contact manifold and $K$ a null-homologous knot in $Y$. 

\begin{enumerate}
  \item $\mathcal{L}(K_{(p,q)})$ contains a Legendrian large cable with $\tb=pq+m$ for $m > 0$ if and only if there is a neighborhood $S$ of $K$ that contains a neighborhood of $\partial S$ that is a length $2m$ balanced continued fraction block of which the center slope is $q/p$.\label{llc-model:existence}
  \item Suppose $(Y,\xi) = (S^3,\xi_{std})$, $L \in \mathcal{L}(K_{(p,q)})$ is a Legendrian large cable with $\tb = pq+m$ for $m > 0$ and $S$ is a neighborhood of $K$. Then there exists a unique length $2m$ balanced continued fraction block up to contact isotopy which is smoothly isotopic to a neighborhood of $\partial S$, has the center slope $q/p$ and contains $L$.\label{llc-model:unique}
\end{enumerate}
\end{theorem}

We call the balanced continued fraction block in Item~(\ref{llc-model:unique}) of Theorem~\ref{llc-model} \dfn{the continued fraction block associated to $L$}. The above theorem identifies the Thurston-Bennequin invariant of a Legendrian large cable, and now we would like to know how to compute its rotation number. 

\begin{figure}[htb]{\tiny
\begin{overpic}
{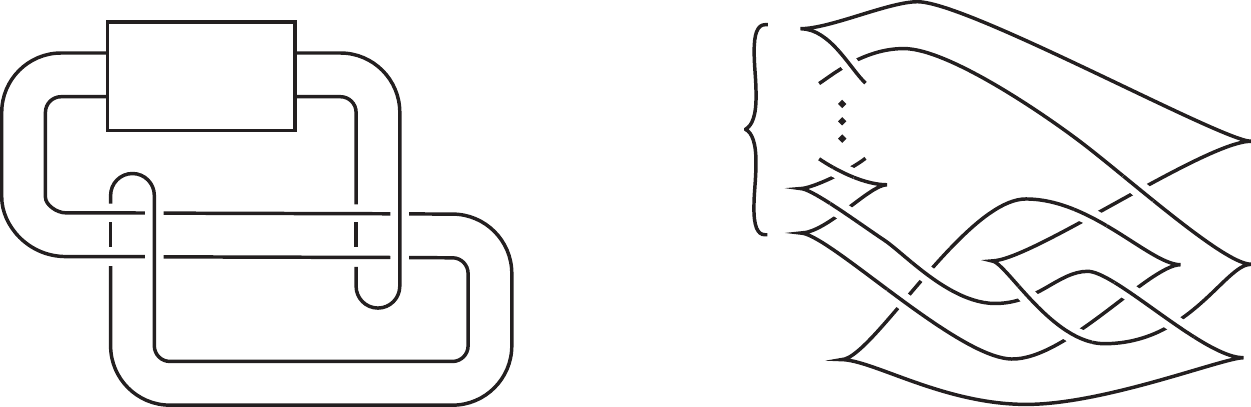}
\put(183, 83){$-m \text{ full}$}
\put(188, 75){twists}
\put(45, 92){\normalsize $m-1$}
\end{overpic}}
\caption{The knot $K_m$ on the left (the box represents $m-1$ full right handed twists). On the right a Legendrian representative of $K_{-m}$.}
\label{Yasui-knots}
\end{figure}

\begin{lemma}\label{llcrot}
Let $B$ be a balanced continued fraction block of length $2m$ with center slope $q/p$ and $L$ be the associated Legendrian large cable. The rotation number of $L$ agrees with the rotation number of any $q/p$ ruling curve on either the front or back face of $B$. 
\end{lemma}

The next theorem shows that there is a unique Legendrian large cable in a fixed balanced continued fraction block when the ambient manifold is $(S^3, \xi_{std})$.

\begin{theorem}\label{llc-isotopy}
  Two Legendrian large cables in $\mathcal{L}(K_{(p,q)})$ in $(S^3, \xi_{std})$ with $\tb = pq + m$ are Legendrian isotopic if and only if the two length $2m$ continued fraction blocks associated to the cables are contact isotopic. 
\end{theorem}
\begin{remark}
The proof of this theorem will show that the ``if" implication is true in any tight contact manifold, but in general only up to contactomorphism. 
\end{remark}

Suppose that $B$ is a balanced continued fraction block of length $2m$ with center slope $q/p$. Inside $B$ are two balanced continued fraction blocks $S_+(B)$ and $S_-(B)$ of length $2(m-1)$, where $B\setminus S_\pm(B)$ consists of two basic slices with the one containing the front face of $B$ being $\pm$ and the one containing the back face of $B$ being $\mp$. There are also $(m+1)$ convex tori $T_i$ with dividing slope $q/p$ where $T_i$ separates $B$ into two thickened tori with the one on the positive side of $T_i$ having $i$ positive basic slices. 
\begin{theorem}\label{stabilizellc}
Let $B$ be a balanced continued faction block of length $2m$ with center slope $q/p$ bounding a solid torus in the knot type $K$ in $(S^3, \xi_{std})$ and  $L$ be the associated Legendrian $(p,q)$-cable of $K$ with Thurston-Bennequin invariant $pq+m$. If $m>1$, then the stabilizations $S_\pm(L)$ of $L$ are the Legendrian large cables of $K$ associated to $S_\pm(B)$. The $m$-fold stabilization $S_+^i(S_-^{m-i}(L))$ of $L$ will be a Legendrian divide on one of the $T_i$.
\end{theorem}

Using Theorem~\ref{llc-model}, we can show there is a universal upper bound on the Thurston-Bennequin invariant of a cable. To state it we first need some notation. Let $\lfloor q/p \rfloor=a_0, ..., a_n=q/p$ be the shortest path in the Farey graph from $\lfloor q/p \rfloor$ clockwise to $q/p$. In the path  $\lfloor q/p \rfloor=a_0, ..., a_n=q/p$ let $k$ be the largest integer such that $a_{n-k}, \ldots, a_n$ is a continued fraction block. We call $k$ the \dfn{length of the tail of $q/p$}. 

\begin{theorem}\label{tbbound}
Let $q/p$ be a rational number that is not an integer. For any knot type $K$ if $q/p\leq \lceil\omega(K)\rceil$, then 
\[
\tb(L)\leq pq+k
\]
for all Legendrian knots $L$ in $\mathcal{L}(K_{(p,q)})$ where $k$ is the length of the tail of $q/p$. If $q/p>\lceil\omega(K)\rceil$, then 
\[
\tb(L)\leq pq+p\, \overline{tb}(K) -q
\]
for all Legendrian knots $L$ in $\mathcal{L}(K_{(p,q)})$.
\end{theorem}

Theorem~\ref{llc-model} also allows us to obtain information about the neighborhood of a knot from the existence of Legendrian large cables. For example, we can improve \cite[Theorem~1.5~and~Proposition~1.7]{Mccullough18}.

\begin{theorem}\label{UTP-Yasui}
  Let $K_m$ be the knots depicted in Figure~\ref{Yasui-knots} with $m \leq -5$. Then they are not uniformly thick in $(S^3, \xi_{std})$ and $\omega(K_m) \geq -\frac{1}{2\left\lfloor \frac{3-m}{4} \right\rfloor - 1}$.
\end{theorem}

In all known examples in $(S^3, \xi_{std})$ Legendrian large cables are always $(n,-1)$-cables and are cables of Lagrangian slice Legendrian knots. So we ask the following questions.
\begin{question}
If $\mathcal{L}(K_{(p,q)})$ contains a Legendrian knot with Thurston-Bennequin invariant greater than $pq$, is $(p,q)=(n,-1)$?
\end{question}
\begin{question}
If $\mathcal{L}(K_{(p,q)})$ contains a Legendrian knot with Thurston-Bennequin invariant greater than $pq$, is there a Legendrian knot in $\mathcal{L}(K)$ that is Lagrangian slice?
\end{question}

\subsection{Classification of negative cables}\label{negcablesec}
We say that a knot type $K$ is \dfn{$s$-minimally thickenable} if any solid torus whose core is in the knot type $K$ that has convex boundary with dividing slope $s$ thickens to a solid torus that has convex boundary with two dividing curves of the same dividing slope.

Let $\lfloor q/p \rfloor=a_0, ..., a_n=q/p$ be the shortest path in the Farey graph from $\lfloor q/p \rfloor$ clockwise to $q/p$, and set $a_i= \lfloor q/p \rfloor + i$ for $i<0$. Let $k$ be the length of the tail of $q/p$ as defined in the last section. Let $a_{n+i}$ be the points that continue the continued fraction block $a_{n-k}, \ldots, a_n$, for $i=1, \ldots, k$. 

\begin{example}
  Suppose $q/p = -12/5$. Then $n=2$, $a_0 = -3$, $a_1 = -5/2$ and $a_2 = -12/5$. Also, $k=1$ and $a_3 = -19/8$. See Figure~\ref{path}.
\end{example}

\begin{figure}[htb]{\tiny
\begin{overpic}[scale=2] 
{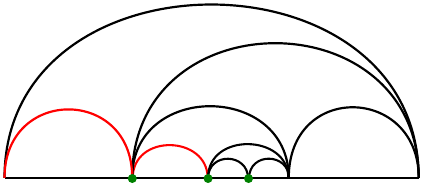}
\put(-5, -7){$-3$}
\put(68, -7){$-\frac{5}{2}$}
\put(110, -7){$-\frac{12}{5}$}
\put(133, -7){$-\frac{19}{8}$}
\put(158, -7){$-\frac{7}{3}$}
\put(235, -7){$-2$}
\end{overpic}}
\vspace{0.1in}
\caption{The clockwise shortest path (red) and the points of the continued fraction block (green) for $-12/5$ in the Farey graph.}
\label{path}
\end{figure}
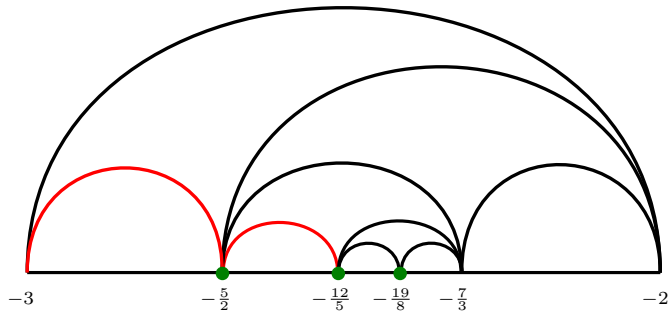

\noindent
{\bf Tori realizing the knot type $K$:}
For $i< n$, let $d_i$ be the number of solid tori (up to contact isotopy) in the knot type $K$ that have convex boundary with two dividing curves of slope $a_i$ and are not contained in a solid torus that has convex boundary with dividing curves of slope $a_{i+1}$. Denote these tori by $N_i^j,$ for $j=1, \ldots, d_i$. 

For $i=n$, let $d_n$ be the number of solid tori (up to contact isotopy) in the knot type $K$ that have convex boundary with two dividing curves of slope $a_n$. Denote these tori by $N_n^j,$ for $j=1, \ldots, d_n$.

For $i=n+1, \ldots, n+k$, let $d_i$ be the number of solid tori (up to contact isotopy) in the knot type $K$ that have convex boundary with two dividing curves of slope $a_i$ and contain a length $2(i-n)$ balanced continued fraction block with center slope $q/p$ and this block does not thicken to a length $2(i+1-n)$ balanced continued fraction block. Denote these tori by $N_i^j,$ for $j=1, \ldots, d_i$.

\noindent
{\bf Standard Legendrian cables:} For $i<n$, let $L_i^j$ be a ruling curve of slope $q/p$ on the boundary of $N_i^j$. For $i=n$, let $L_n^j$ be a Legendrian divide on the boundary of $N_n^j$. For $i=n+1, \ldots, n+k$, let $L_i^j$ be the Legendrian large cable given by Theorem~\ref{llc-model} and the balanced continued fraction block in $N_i^j$ of length $2(i-n)$.

\noindent
{\bf Commensurating tori:} Given $N_i^j$ and $N_{i'}^{j'}$ we call a torus $N$ with convex boundary and two dividing curves of slope $a_m$ where $m<\min \{i, i'\}$ a \dfn{commensurating torus for $N_i^j$ and $N_{i'}^{j'}$} if, after contact isotopy, $N\subset N_i^j\cap N_{i'}^{j'}$. We say $N$ is a \dfn{maximal commensurating torus} if there is not another commensurating torus $N'$ having dividing slope larger than $N$ and with $N\subset N'\subset N_i^j\cap N_{i'}^{j'}$. Given $i,i', j, j'$ we denote by $C_{i, i' j, j'}$ the set of all maximal commensurating tori for $N_i^j$ and $N_{i'}^{j'}$.

\noindent
{\bf Super commensurating tori:} Given $N_n^j$ and $N_n^{j'}$ we call a torus $N$ with convex boundary and two dividing curves a \dfn{super commensurating tori for $N_n^j$ and $N_n^{j'}$} if, after contact isotopy, $N_n^j\cup N_n^{j'}\subset N$. We say $N$ is a \dfn{minimal super commensurating torus} if there is no other super commensurating tori $N'$ having dividing slope less than $N$ and with $N_n^j\cup N_n^{j'}\subset N'\subset N$. Given $j, j'$ we denote by $C'_{j,j'}$ the set of all minimal super commensurating tori for $N_n^j$ and $N_n^{j'}$.

We define the \dfn{cone of $L$} as 
\[
C(L)=\{S^k_+S^l_- (L) : \text{ for all } k \text{ and } l\}.
\]

In the statement of the theorem we use the notations $a/b\ominus c/d= (a-c)/(b-d)$ and $a/b\bigcdot c/d= ad-bc$. 

\begin{theorem}\label{classification-negative} 
Using the notation established above,  suppose $K$ is a $q/p$-minimally thickenable knot in $(S^3,\xi_{std})$. If $q/p < \omega(K)$ then
\begin{enumerate}
\item\label{neg1} All Legendrian knots in $\mathcal{L}(K_{(p,q)})$ destabilize to a standard Legendrian cable and all the standard Legendrian cables with $i\not= n$ do not destabilize. If $i=n$ then $L_n^j$ destabilizes if and only if $\partial N_n^j$ is contact isotopic to a torus in one of the $N_i^{j'}$ with $i>n$. 
\item\label{neg2} The standard cables $L_i^j$ and $L_{i'}^{j'}$ are Legendrian isotopic if and only if $i=i'$ and $j=j'$.
\item\label{neg3} The Thurston-Bennequin invariants of the standard Legendrian cables are
\[
\tb(L_i^j) = \begin{cases}
 pq - |a_i \bigcdot q/p| & i<n\\
pq& i=n\\
pq+(i-n) & i>n.
\end{cases}
\]
\item\label{neg4} If $i\leq 0$, the rotation number of $L_i^j$ is 
\[
  \rot(L_i^j) = p \rot(L)
\]
where $L \in \mathcal{L}(K)$ is a core of $N_i^j$. If $0\leq i\leq n$, then the rotation number of $L_i^j$ is determined as follows: Factor $N_i^j$ into $N \cup B_1 \cup \dots \cup B_i$ where each $B_l$ is a basic slice with dividing slopes $a_{l-1}$ and $a_l$. Let $\epsilon_l$ be the sign of the basic slice $B_l$ and $L \in \mathcal{L}(K)$ be a core of $N$. Then, 
\[
  \rot(L_i^j) = p \rot(L) - \sum_{l=1}^i {\epsilon_l} |(a_{l}\ominus a_{l-1})\bigcdot q/p|.
\]
If $i>n$, then the rotation number of $L_i^j$ is determined by computing the rotation number of a ruling curve of slope $q/p$ on $\partial N_i^j$.
\item\label{neg5} Given $i, i'\leq n$, then for each maximal commensurating torus $N\in C_{i,i',j,j'}$ with dividing slope $a_m$ we have 
    \[
      L = S_+^{k}S_-^{l}(L_i^j) = S_+^{k'}S_-^{l'}(L_{i'}^{j'}),
    \]
where $k+l = |(a_i \ominus a_m) \bigcdot q/p|$ and  $k'+l' = |(a_{i'} \ominus a_m)\bigcdot q/p|$.
Moreover, decompose $N_i^j \setminus N$ into basic slices $B_{m+1}, \ldots ,B_i$ where $B_u$ is a basic slice with dividing slopes $a_{u-1}$ and $a_{u}$. Now let $c_1, \ldots, c_y$ be a subsequence of indices $m+1,\ldots, i$ corresponding to positive basic slices. Then,
\[
  k = \sum_{u=1}^y |(a_{c_u} \ominus a_{c_u-1})\bigcdot q/p|.
\]
The $l$, $k'$ and $l'$ are determined in a similar way.
\item\label{neg6} Given $j,j'$, then for each minimal super commensurating torus $N\in C'_{j,j'}$ with dividing slope $s$ we have 
    \[
      L = S_+^{k}S_-^{l}(L_n^j) = S_+^{k'}S_-^{l'}(L_{n}^{j'}),
    \]
where $k+l = k'+l' = |s \bigcdot q/p|$. Moreover, decompose $N \setminus N_n^j$ into basic slices $B_1, \ldots ,B_m$ where $B_u$ is a basic slice with dividing slopes $b_{u-1}$ and $b_{u}$ such that $q/p=b_0,...,b_m=s$ is the shortest path in the Farey graph clockwise from $q/p$ to $s$. Now let $c_1, \ldots, c_y$ be a subsequence of indices $1,\ldots, m$ corresponding to negative basic slices. Then,
\[
  k = \sum_{u=1}^y |(b_{c_u}\ominus b_{c_u-1}) \bigcdot q/p|.
\]
The $l$, $k'$ and $l'$ are determined in a similar way.  
\item\label{neg7} Given $i,i'\leq n$, then the stabilizations $S_+^{k}S_-^{l}(L_i^j)$ and $S_+^{k'}S_-^{l'}(L_{i'}^{j'})$ remain distinct unless they are related by a sequence of equivalences from Items~\eqref{neg5} and~\eqref{neg6} (and further stabilizations). 
\item\label{neg8} Given $i>n$, then $N_i^j$ contains $(i-n)$ different $N_n^{j'}$ and each possible way of stabilizing $L^j_i$, $(i-n)$ times will give $L_n^{j'}$ for some $j'$.
\item\label{neg9} Given $i,i'>n$, then a stabilization of $L^j_i$ and $L^{j'}_{i'}$ that keeps their Thurston-Bennequin invariants above $pq$ will be Legendrian isotopic if and only if their associated balanced continued fraction blocks (from Theorem~\ref{llc-model}) are contact isotopic.
\item\label{neg10} The set of Legendrian knots $\mathcal{L}(K_{(p,q)})$ is 
  \[
        \bigcup_{i \leq n+k} C(L_i^j) / \sim,
    \] 
    where the equivalence relation $\sim$ is induced from \eqref{neg5}, \eqref{neg6}, \eqref{neg7}, \eqref{neg8}, and  \eqref{neg9}.
\end{enumerate}
\end{theorem}

\begin{remark}
We note that in all prior work studying cables, or more generally satellites, one needed some form of uniform (or partial uniform) thickenability for $K$ to obtain any classification result. However in this theorem we only need to know the $K$ is $q/p$-minimally thickenable, a significantly weaker condition. In addition, instead of needing to know all non-thickenable and partially thickenable tori in the knot type $K$, as was done in the past, we only need to know about certain tori with dividing slope $a_i$ that do not thicken to $a_{i+1}$. 
\end{remark}
\begin{remark}
Notice that the classification of transverse knots in the knot type $K_{(p,q)}$ will follow from the above theorem. 
\end{remark}
\begin{remark}
We note that almost all previous results about cables less than $\omega(K)$ follow from this theorem. For example both \cite{EH04, Tosun13} easily follow. Also the classification of cables of the positive trefoil from \cite{ELT12} follows. In \cite{ELT12} they also classify the cables of the $(r,s)$-torus knot with cabling slope $q/p\in (0, rs-r-s)$. All these results follow from the above theorem (and the understanding of tori from \cite{ELT12}) except for cabling slopes $(rs-r-s)/n$ where $n$ is not relatively prime to $rs-r-s$. It is not hard to see how to adapt Theorem~\ref{classification-negative} to this situation, but its statement would become even more unwieldy.
\end{remark}

If $\omega(K)$ is not an integer, we have an additional classification result.

\begin{theorem}\label{classification-in-between}
Suppose $K$ is a knot in $(S^3,\xi_{std})$ with $\omega(K)\leq q/p<\lceil \omega(K) \rceil$. Let $m$ be the maximum number such that $a_m \leq \omega(K)$. Notice that $m\leq n$. Then, 
\begin{enumerate}
\item\label{btw1} All Legendrian knots in $\mathcal{L}(K_{(p,q)})$ destabilize to a standard Legendrian cable and all the standard Legendrian cables do not destabilize. 
\item\label{btw2} The standard cables $L_i^j$ and $L_{i'}^{j'}$ are Legendrian isotopic if and only if $i=i'$ and $j=j'$. 
\item\label{btw3} The Thurston-Bennequin invariants of the standard Legendrian cables are
\[
\tb(L_i^j) =
  pq - |a_i \bigcdot q/p| 
\]
\item\label{btw4} The rotation number of $L_i^j$ is 
\[
  \rot(L_i^j) = p \rot(L)
\]
if $i\leq 0$. If $0\leq i\leq m$, then the rotation numbers of $L_i^j$ are determined as follows: Factor $N_i^j$ into $N \cup B_1 \cup \dots \cup B_i$ where each $B_l$ is a basic slice with dividing slopes $a_{l-1}$ and $a_l$. Let $\epsilon_l$ be the sign of the basic slice $B_l$ and $L \in \mathcal{L}(K)$ be a core of $N$. Then, 
\[
  \rot(L_i^j) = p \rot(L) - \sum_l{\epsilon_l} |(a_{l}\ominus a_{l-1})\bigcdot q/p|.
\]
\item\label{btw5} Given $i, i'\leq m$, then for each maximal commensurating torus $N\in C_{i,i',j,j'}$ with dividing slope $a_m$ we have 
\[
  L = S_+^{k}S_-^{l}(L_i^j) = S_+^{k'}S_-^{l'}(L_{i'}^{j'}),
\]
where $k+l = |(a_i \ominus a_m) \bigcdot q/p|$ and  $k'+l' = |(a_{i'} \ominus a_m) \bigcdot q/p|$. Moreover, decompose $N_i^j \setminus N$ into basic slices $B_{m+1}, \ldots ,B_i$ where $B_u$ is a basic slice with dividing slopes $a_{u-1}$ and $a_{u}$. Now let $c_1, \ldots, c_y$ be a subsequence of indices $m+1,\ldots, i$ corresponding to positive basic slices. Then,
\[
  k = \sum_{s=1}^y |(a_{c_u} \ominus a_{c_u-1})\bigcdot q/p|.
\]
The $l$, $k'$ and $l'$ are determined in a similar way.
\item\label{btw6}  The stabilizations $S_+^{k}S_-^{l}(L_i^j)$ and $S_+^{k'}S_-^{l'}(L_{i'}^{j'})$ remain distinct unless $(k,j)$ and $(k',j')$ come from (5), or a further stabilization of, the stabilizations from Item~\eqref{btw5}.
\item The mountain range of $\mathcal{L}(K_{(p,q)})$ is
\[
  \bigcup_{i \leq m} C(L_i^j) / \sim,
\] 
where the equivalence relation $\sim$ is induced from \eqref{btw5} and \eqref{btw6}. \label{btw7}
\end{enumerate}
\end{theorem}

There are more results which were inaccessible with previous work as the knot type contains Legendrian large cables or was not known to be uniformly thick. In work to appear, \cite{BakerEtnyreMinOnaran}, the second and third author together with Baker and Onaran will present a classification of Legendrian and transverse torus knots in tight Lens spaces. Also in \cite{Min20}, the third author will give a classification of negative cables of some twist knots.

\bigskip
\noindent
{\bf Organization:} In Section~\ref{background} we recall some background results and prove a few preliminary lemmas. In Section~\ref{pcc} we give the proofs of the theorems for positive cables, that is Theorems~\ref{classification-positive} and~\ref{postransverse}. All the results about Legendrian large cables in Section~\ref{introllc} are established in Section~\ref{llcsec}. In Section~\ref{negsec} we prove all the results about negative cables, more specifically we prove Theorems~\ref{classification-negative} and~\ref{classification-in-between}. Finally, in Section~\ref{grc} we verify the results in Section~\ref{introgrc} concerning the general behavior of Legendrian knots in the knot type of positive and negative cables. 

\bigskip
\noindent
{\bf Acknowledgements:} We thank the referees for many valuable comments on the first version of the paper. The second and third authors were partially supported by the NSF grants DMS-1608684 and DMS-1906414.

\section{Background}\label{background}

We assume that the reader has a basic understanding of $3$--dimensional contact geometry, including convex surface theory, Legendrian knots, and their invariants. We recall several definitions and theorems about contact structures on $3$--manifolds and Legendrian knots that we will use frequently. For more details, see \cite{Etnyre:convex, Etnyre:LT, Geiges:book, Honda00}. Also, for more details and figures on the Farey graph and continued fractions, see \cite[Section~2]{EtnyreRoy20}.

In Subsection~\ref{farey} we describe the Farey graph and discuss curves on tori. In the following subsection we recall the notation of a bypass and its effects on the dividing curves of tori. 
Then in Subsection~\ref{cfb} we review the classification of contact structures on solid tori and thickened tori. Subsection~\ref{compclass} discusses the computations of classical invariants of cabled knots as well as some results about intersections between curves on tori. Finally, in Subsection~\ref{cabtor} we will show that any two essential annuli in the complement of a cabled knot are smoothly isotopic. 

\subsection{Curves on tori and the Farey graph} \label{farey}
We will keep track of curves on a torus using the Farey graph. First recall that embedded curves on $T^2$ are in one to one correspondence with the rational numbers union $\infty$. We described our convention for this correspondence in the beginning of the introduction.  

Consider the unit disk in $\R^2$ with the hyperbolic metric on its interior. Label the point $(0,1)$ by $0/1$ and $(0,-1)$ by $\infty=1/0$ and connect them with a hyperbolic geodesic. Now consider points on the boundary of the unit disk with positive $x$-coordinates. Given two points that have been labeled already, say by $a/b$ and $c/d$, label the midpoint between them on the boundary of the unit disk by $(a+c)/(b+d)$ and then connect this new point to the two other points by hyperbolic geodesics. We will denote $(a+c)/(b+d)$ by $a/b \oplus c/d$. Iterate this until all of the positive rational numbers appear as a label. Now label the points on the boundary of the unit disk with $x$-coordinate negative in the same manner except now think of the point $(0,-1)$ labeled $\infty$ as $-1/0$. 

Given two vertices $r$ and $s$ in the Farey graph we will denote by $[r,s]$ all the vertices that are clockwise of $r$ and counterclockwise of $s$ (and similarly for $(r,s]$, $(r,s)$, and $[r,s)$).

It is useful to know that two curves on $T^2$ form a basis for the homology of $T^2$ if and only if they are connected by an edge in the Farey graph and this is true if and only if they have representatives that intersect (transversely) exactly once. 

We also note for future use that the minimal number of times curves of slope $a/b$ and $c/d$ will intersect is $|ad-bc|$. We denote $ad-bc$ by $a/b \bigcdot c/d$. 

\subsection{Bypasses} \label{bypasses}
A key tool we will use in our work is bypasses. Everything discussed here comes from \cite{Honda00}. Suppose $\Sigma$ is a convex surface in a contact manifold $(Y,\xi)$. A \dfn{bypass} for $\Sigma$ is a disk $D$ in $Y$ whose boundary consists of two arcs $\alpha_1$ and $\alpha_2$ such that 
\begin{enumerate}
\item $D\cap \Sigma=\alpha_1$ and the intersection of $D$ and $\Sigma$ is transverse.
\item $\alpha_1$ is a Legendrian arc in $\Sigma$ that intersects the dividing curves of $\Sigma$ in its two end points and in one interior point.
\item The characteristic foliation of $D$ has elliptic singularities along $\alpha_1$ where $\alpha_1$ intersects the dividing curves $\Gamma_\Sigma$ and the signs of the singular points alternate. These are the only singularities along $\alpha_1$. 
\item $\alpha_2$ is Legendrian and the singularities of $D$ along $\alpha_2$ all have the same sign. 
\end{enumerate}
If $\Sigma$ is oriented and $D$ is on the positive side of $\Sigma$ then if one pushes $\Sigma$ past $D$ to get a surface $\Sigma'$, we say $\Sigma'$ is the result of \dfn{attaching the bypass} to $\Sigma$. One may also attach a bypass to $\Sigma$ from the negative side. We can think of $\Sigma$ and $\Sigma'$ agreeing except along a small disk neighborhood of $\alpha_1$ or as being disjoint and cobounding a $\Sigma\times [-0,1]$. 

The dividing curves of $\Sigma'$ are related to those on $\Sigma$ in a prescribed way, see \cite[Section~3.4]{Honda00}, but we will only discuss the case when $\Sigma$ is a torus. In this case $\Sigma$ will have an even number of parallel dividing curves of some slope $s$ (we are assuming the contact structure is tight). Now when attaching a bypass to $\Sigma$ one of three things will happen:
\begin{enumerate}
\item the number of dividing curves on $\Sigma$ will increase (and the slope will stay the same),
\item the number of dividing curves on $\Sigma$ will decrease (and the slope will stay the same), or
\item the slope of the dividing curves will change and this can happen only if there are two dividing curves. 
\end{enumerate}

In certain situations we can say more. Given $\Sigma$ with $2n$ dividing curves $\Gamma_\Sigma$ of slope $s$. One can use Giroux flexibility \cite{Giroux91} to arrange that the characteristic foliation has $2n$ lines of singular points called \dfn{Legendrian divides} parallel to $\Gamma_\Sigma$ and there is one such line in each component of $\Sigma\setminus \Gamma_\Sigma$. The rest of the foliation is linear of slope $r$. We can choose the slope $r$ to be any slope except $s$. These curves in this linear foliation are called \dfn{ruling curves}. We say the foliation of $\Sigma$ is in \dfn{standard form} if it is as described above. 

\begin{lemma}\label{bypassattach}
Suppose that $\Sigma$ is a convex torus in standard form with ruling curves of slope $r$ and dividing curves of slope $s$. If a bypass is attached to the front of $\Sigma$ along a ruling curve, then the number of dividing curves will decrease if there are more than two dividing curves and if there are two dividing curves then the slope of the dividing curves will change to $s'$ where $s'$ is the point in $[s,r]$ that is closest to $r$ with an edge in the Farey graph to $s$. If the bypass is attached from the back side of $\Sigma$, then the same thing happens to the dividing curves, except in the case of two dividing curves one must consider the interval $[r,s]$ instead of $[s,r]$. 
\end{lemma}

We also recall a standard way to find bypasses. If $\Sigma$ and $\Sigma'$ are disjoint convex surfaces and $A$ is a convex annulus with one boundary component a Legendrian curve on $\Sigma$ and the other boundary component a Legendrian curve $L$ on $\Sigma'$ but otherwise disjoint from $\Sigma\cup \Sigma'$, then if the boundary of $A$ intersects the dividing curves of $\Sigma$ more than those of $\Sigma'$, then there will be a dividing curve on $A$ that cobounds a disk with a portion of $L$. One may use Giroux flexibility to use this disk to build a bypass for $\Sigma$. 

\subsection{Contact structures on \texorpdfstring{$S^1\times D^2$}{S1xD2} and \texorpdfstring{$T^2\times [0,1]$}{T2x[0,1]}} \label{cfb}

Here we briefly recall the classification of contact structures on solid tori and thickened tori due to Giroux \cite{Giroux00} and Honda \cite{Honda00}. 

Consider a contact structure $\xi$ on $T^2\times[0,1]$ that has convex boundary with dividing curves of slope $s_0$ on $T^2\times \{0\}$ and slope $s_1$ on $T^2\times \{1\}$. We also assume that each boundary component has two dividing curves. We say $\xi$ is \dfn{minimally twisting} if any convex torus in $T^2\times [0,1]$ that is parallel to $T^2\times \{0\}$ has dividing slope in $[s_0,s_1]$. Now consider a minimal path in the Farey graph from $s_0$ clockwise to $s_1$. We call this a decorated path if each edge has been assigned a $+$ sign or a $-$ sign. We call a path in the Farey graph a \dfn{continued fraction block} if there is a change of basis such that the path goes from $0$ clockwise to $n$ for some positive integer $n$. We say two choices of signs on a continued fraction block are related by \dfn{shuffling} if they have the same number of $+$ signs (and hence the same number of $-$ signs too). 

\begin{theorem}[Giroux and Honda, 2000 \cite{Giroux00, Honda00}] \label{basic-slice}
Each decorated minimal path in the Farey graph from $s_0$ clockwise to $s_1$ describes a minimally twisting contact structures on $T^2 \times [0,1]$ with two dividing curves on each boundary component of slopes $s_0$ and $s_1$. Two such decorated paths will describe the same contact structure if and only if the decorations differ only by shuffling in continued fraction blocks. 
\end{theorem}

A particular case of Theorem~\ref{basic-slice} is when $s_0$ and $s_1$ share an edge in the Farey graph. In this case the theorem says that there are exactly two minimally twisting contact structures. These are called \dfn{basic slices} and the above theorem says that all minimally twisting contact structures can be thought of as stacking several basic slices together. The two different contact structures on a basic slice can be distinguished by their relative Euler class, and after picking an orientation, we call them \dfn{positive} and \dfn{negative basic slices}. We also define $\epsilon(B)$ to be $1$ if $B=(T^2\times [0,1], \xi)$ is a positive basic slice, $-1$ if $B$ is a negative basic slice. We note that the result of attaching a bypass to a torus $T$ in standard form with two dividing curves is another torus $T'$ that cobounds with $T$ a basic slice. 

One can compute the relative Euler class of a minimally twisting contact structure on $T^2\times [0,1]$ with dividing slopes $s_0$ and $s_1$ as follows. Let $s_0=a_1,\ldots, a_n=s_1$ be the vertices in a minimal path from $s_0$ to $s_1$ and let $\epsilon_i$ be the sign on the basic slice corresponding to $s_{i-1}$ and $s_i$. Then the relative Euler class of the contact structure corresponding to this path is Poincar\'e dual to the curve
\[
\sum_{i=1}^n \epsilon_i(a_i\ominus a_{i-1}) 
\]
where $a/b\ominus c/d= (a-c)/(b-d)$. 

The following two lemmas are direct consequences of the classification of tight contact structures on $T^2\times I$ in \cite{Giroux00, Honda00}.

\begin{lemma}\label{realizeslopes}
If $\xi$ is a minimally twisting contact structure on $T^2\times I$ with boundary slopes $s_0$ and $s_1$, then any slope $s\in[s_0,s_1]$ can be realized by a convex torus parallel to the boundary with two dividing curves.
\end{lemma}

\begin{lemma}\label{anyslope}
If $\xi$ is not a minimally twisting contact structure on $T^2\times I$ then any slope may be realized as dividing curves on a convex torus in $T^2\times I$ parallel to the boundary.
\end{lemma}

A contact structure $\xi$ on $T^2\times [0,1]$ described by a continued fraction block of edges in the Farey graph will also be called a \dfn{continued fraction block}. We say that it is \dfn{balanced} if it has the same number of positive and negative signs in the decorated path describing the contact structure. Suppose the length of the path describing the contact structure is $2n$, then after a coordinate change we can assume that the slope of the dividing curves on the back face is $-n$ and on the front face is $n$. 
Let $A$ be an annulus with slope $0$ in this balanced continued fraction block. Then the relative Euler class of the contact structure evaluated on $A$ is $0$.

Now turning to contact structures on solid tori $S^1\times D^2$ with our slope convention, recall that the slope of the meridian is $\infty$. Call a path in the Farey graph \dfn{almost decorated} if a sign has been assigned to all but the counterclockwise most edge. 

\begin{theorem}[Giroux and Honda, 2000 \cite{Giroux00, Honda00}] 
Each almost decorated minimal path in the Farey graph from $-\infty=\infty$ clockwise to $s$ describes a contact structures on $S^1\times D^2$ with two dividing curves on the boundary  of slopes $s$. Two such decorated paths will describe the same contact structure if and only if the decorations differ only by shuffling in continued fraction blocks. 
\end{theorem}

We will also need the following result. 

\begin{lemma}\label{anyslopeless}
If $\xi$ is a tight contact structure on the solid torus $S^1\times D^1$ so that the boundary is convex with dividing slope $s$, then any slope less than or equal to $s$ may be realized as the dividing slope on a convex torus parallel to the boundary (and we may assume the torus has two dividing curves). 
\end{lemma}

\subsection{Knots in contact manifolds}\label{kicm}
All facts in this section can be found in \cite{EtnyreHonda01b}.

A null-homologous Legendrian knot $L$ in a contact manifold $(Y,\xi)$ has a neighborhood $N$ with convex boundary having two dividing curves of slope $\tb(L)$. If $\partial N$ is in standard form, then we say that $N$ is a \dfn{standard neighborhood} of $L$. 

Using a model for the standard neighborhood of $L$ we can take a vector field $v$ along $L$ that is tangent to $\xi$ and transverse to $L$. Pushing $L$ along $\pm v$ will result in a transverse knot. If these knots are oriented by $\xi$ and $L$ has an orientation, then one of the transverse knots will have an orientation that agrees with $L$. We call this the \dfn{positive transverse push-off} and denote it $L_+$. The other transverse knot is the \dfn{negative transverse push-off} and is denoted $L_-$. One may easily compute that $sl(L_+)= \tb(L)-r(L)$. 

We also have the following useful fact, which is an immediate consequence of Theorem~2.4.2 in \cite{Eliashberg92}. 
\begin{theorem} \label{Legendrian-isotopy}
Legendrian knots in $(S^3,\xi_{std})$ are Legendrian isotopic if and only if there is a contactomorphism of $S^3$ taking one of the knots to the other. 
\end{theorem}

If $T$ is a transverse knot in $(Y,\xi)$ then it has a neighborhood $N$ that is contactomorphic to $(S^1\times D_a^2, \ker(d\phi + r^2\, d\theta))$ for some small $a$ where $D^2_a$ is the disk of radius $a$ in $\R^2$. One may easily check that for any integer $n$ large enough there will be a torus parallel to $\partial N$ inside of $N$ with linear characteristic foliation of slope $-n$. Let $T_n$ be one of the leaves in the characteristic foliation. Clearly $T_n$ is a Legendrian knot that is smoothly isotopic to $T$. We call $T_n$ a \dfn{Legendrian approximation} to $T$. Notice that there are infinitely many Legendrian approximations of $T$. One may show that $(T_n)_+$ is transversely isotopic to $T$. It is also easy to check that $\tb(T_n)=-n$ and $r(T_n)=-n-sl(T)$. We also know that $T_{n+1}=S_-(T_n)$ where $S_-(L)$ is the negative stabilization of $L$. 

\begin{theorem}[Etnyre and Honda 2001, \cite{EtnyreHonda01b}]\label{traniso}
Two transverse knots $T$ and $T'$ are transversely isotopic if and only if they have Legendrian approximations that become Legendrian isotopic after a suitable number of negative stabilizations. 
\end{theorem}

The \dfn{contact width} $\omega(K)$ of a knot is the supremum of dividing slopes of convex tori representing the knot type $K$. We say that $K$ has the \dfn{uniform thickness property} if
\begin{itemize}
  \item any solid torus representing the knot type $K$ can be thickened to a standard neighborhood of $L$, a Legendrian representative of $K$ with $\tb(L)=\tbb(K)$, and
  \item $\omega(K)$ is equal to $\tbb(K)$.
\end{itemize}  

We end this section by discussing the relation between bypasses and stabilization. Suppose that $N$ is a standard neighborhood of a Legendrian knot $L$. If the ruling slope of $N$ is larger than $\tb(L)+1$ and there is a bypass for $\partial N$ along the ruling curve, then attaching it will result in a new solid torus that bounds a solid torus $N'$ and $\partial N'$ will have dividing slope $\tb(L)+1$. The torus $N'$ is a neighborhood of a unique Legendrian knot $L'$ and it is not hard to show that $L$ is a stabilization of $L'$. The sign of the stabilization corresponds to the sign of the bypass.  

We have a similar way to destabilize a Legendrian knot that will be used frequently in the work below. Let $L$ be a Legendrian knot sitting on a convex surface $\Sigma$. If there is a disk embedded in $\Sigma$ with boundary an arc on $L$ and an arc on the dividing set of $\Sigma$ that is otherwise disjoint from $L$ and $\Gamma_\Sigma$, then one may use the disk to isotope $L$ to a curve $L'$ on $\Sigma$ that may be Legendrian realized. The knot $L'$ is a destabilization of $L$. In fact the disk on $\Sigma$ can be turned into a bypass for a neighborhood of $L$.

\subsection{Computation of classical invariants}\label{compclass}
When writing a rational number $q/p$, we will always take $p$ to be positive. (Recall in the introduction we mentioned that when considering cables we always assume that $p$ is positive since if $p$ is negative then we would consider this to be a cable of the knot with the reversed orientation.)

\begin{lemma}[Etnyre and Honda 2004, \cite{EH04}]\label{crot}
Let $N$ be a solid torus with convex boundary having two dividing curves of slope $s$. Suppose that $L$ is a $(p,q)$-curve on $\partial N$ that is either a Legendrian divide or a ruling curve. In an $I$-invariant neighborhood of $\partial N$ we can arrange that the ruling curves are meridional. Let $\mu$ be one of these curves. We can also arrange that the Seifert longitude $\lambda$ is a ruling curve or Legendrian divide on another copy of $\partial N$ in the $I$-invariant neighborhood. The rotation number of $L$ is computed by
\[
r(L)=p\, r(\lambda) + q\, r(\mu). 
\]
\end{lemma}

\begin{lemma}[Etnyre and Honda 2004, \cite{EH04}]\label{ctb}
Suppose that $N$ is a solid torus with convex boundary having two dividing curves of slope $s$. Let $L$ be a $(p,q)$-curve on $\partial N$ that is either a Legendrian divide or a ruling curve, then
\[
\tb(L)= pq-|s\bigcdot q/p|.
\]
\end{lemma}

From above we see that it will be useful to compute the number of intersections between two curves. We discuss bounds on such intersections in the next few lemmas. 

\begin{lemma}\label{signs}
  When considering rational numbers we will assume their denominators are positive. 
  Given any rational number $q/p$, then we have $s/r \bigcdot q/p \geq 0$ if $s/r \geq q/p$, and $s/r \bigcdot q/p< 0$ if $s/r < q/p$.
\end{lemma}
\begin{proof}
If $s/r \geq q/p$, we have $sp \geq rq$. Then
\[
  s/r \bigcdot q/p = sp - rq \geq 0.
\]
Similarly, If $s/r < q/p$, we have $sp < rq$. Then
\[
  s/r \bigcdot q/p = sp - rq < 0.
\]
\end{proof}

\begin{lemma}\label{intersections}
If given two integers $k,l$ such that $k<l< q/p$, then 
\[
|k\bigcdot q/p| > |l\bigcdot q/p|.
\]
Moreover, if $s/r\in (k-1,k)$ and $k<q/p$, then 
\[
|s/r\bigcdot q/p| > |(k-1)\bigcdot q/p|>|k\bigcdot q/p|.
\]
\end{lemma}
\begin{proof}
Write $q=pm+r$ with $r\in [0,p-1]$. Now 
\[
k\bigcdot q/p= kp-mp -r= (k-m)p -r< (l-m)p-r=l\bigcdot q/p,
\]
and from Lemma~\ref{signs} these numbers are negative. Thus the first inequality is established. 

Since $s/r$ is strictly between $k-1$ and $k$, we know by the construction of the Farey tessellation that there are positive integers $a$ and $b$ such that 
\[
  s/r = \frac{a(k-1)}{a} \oplus \frac{bk}{b}.
\]
Now we have 
\begin{align*}
  |s/r\bigcdot q/p| &= a|(k-1) \bigcdot  q/p| + b |k \bigcdot q/p|\\
  & > |(k-1) \bigcdot q/p|\\
  & > |k \bigcdot q/p|.
\end{align*}
The second line follows from the fact that $a$ and $b$ are both positive, and the third line follows from the first inequality established above.
\end{proof}

\begin{remark}\label{genintersect}
By changing coordinates on the torus one can generalize the lemma as follows. Given $q/p$ let $\lfloor q/p \rfloor=a_0, a_1, \ldots, a_n=q/p$ be the shortest path in the Farey graph from $\lfloor q/p \rfloor$ clockwise to $q/p$. We have $|a_{i-1}\bigcdot q/p|>|a_i\bigcdot q/p|$ and if $s\in (a_{i-1},a_i)$ then $|s\bigcdot q/p|\geq|a_{i-1}\bigcdot q/p|>|a_i\bigcdot q/p|$.
\end{remark}


\subsection{Cabling tori}\label{cabtor}
We will need to understand annuli in the complement of knots in $S^3$. To this end we first recall the following well-known result.
\begin{lemma}\label{annuli}
Let $K$ be a knot in $S^3$ and let $X_K$ denote the complement of a neighborhood $N$ of $K$. If $A$ is an essential annulus in $X_K$ then either
\begin{enumerate}
\item $K$ is a connected sum $K_1\# K_2$ and $A$ is the the intersection of the connect sum sphere with $X_K$ and $A\cap N$ has slope $\infty$ on $\partial N$, or
\item $K$ is a $(p,q)$-cable of some knot $K'$ and $A$ is the intersection of the boundary of a neighborhood of $K'$ containing $K$ with $X_K$ and $A\cap N$ has slope $pq$ on $\partial N$. 
\end{enumerate} 
\end{lemma}
While well-known, we cannot find a reference that contains all the details in the statement above, so we provide a sketch of the proof.
\begin{proof}
  Let $A_1$ and $A_2$ be the annuli that $\partial A$ breaks $\partial N$ into and suppose $A \cap N$ is a pair of meridians on $\partial N$. Then we can construct a sphere in $S^3$ by gluing two meridional disks and $A$ together. This sphere separates $A_1$ and $A_2$ so it is clear that $K$ is a composite knot. 

  Now suppose $A \cap N$ is a pair of longitudes on $\partial N$. Since any torus in $S^3$ bounds a solid torus, we know that the union of $N$ and a neighborhood $A$ is a thickened torus and bounds a solid torus. Let $K'$ be a core of this solid torus. Clearly $K$ is a cable of $K'$ and $A$ is as claimed in the lemma. 

  Suppose $A \cap N$ is a pair of non-trivial cables on $\partial N$. In this case, the union of $N$ and a neighborhood of $A$ is not a thickened torus. Instead, it is a Seifert fibration over an annulus with a singular fiber. Since each boundary component of this fibration bounds a solid torus, $S^3$ admits a Seifert fibration over a sphere with three singular fibers. If one of these singular fibers is trivial, then $A$ becomes boundary-parallel, which is a contradiction. If all three singular fibers are non-trivial however, the fundamental group of the Seifert fibration is non-trivial (\textit{cf.}~\cite{Brin07}), which implies that the manifold cannot be $S^3$.  


\end{proof}

\begin{lemma}\label{smooth-isotopy}
  Let $T_1$ and $T_2$ be two tori in $S^3$ that bound solid tori whose core is in the knot type $K$. Suppose that $K_{(p,q)}$ is a $(p,q)$-cable of $K$ that lies on both $T_1$ and $T_2$ with $p\not= 1$ (that is the cable is not isotopic to $K$). Then there is a smooth isotopy from $T_1$ to $T_2$ fixing $K_{(p,q)}$.
\end{lemma}
\begin{proof}
We can isotope $T_1$ and $T_2$ so they agree in a neighbourhood of $K_{(p,q)}$ and take a small neighbourhood $N$ of $K_{(p,q)}$ that intersects $T_1$ and $T_2$ where they agree. Now let $A_1$ and $A_2$ be the part of $T_1$ and $T_2$, respectively, outside of the interior of $N$. We see that $A_1$ and $A_2$ agree near $\partial N$ and we will show that we can isotope them to agree everywhere. 

Let $N'$ be a solid torus that is slightly larger than the one bounded by $T_1$, so that $N'$ contains $N$. We can assume that $A_2$ is transverse to $\partial N'$. So $A_2\cap \partial N'$ consists of simple closed curves that consist of parallel homologically essential curves and some homologically non-essential curves on $\partial N'$. Since $A_2$ is essential in $S^3\setminus N$ we know that the homologically non-essential curves on $\partial N'$ are also non-essential on $A_2$. So any homologically non-essential curve bounds a disk on both $\partial N'$ and $A_2$. The union of these disks bounds a ball in $S^3$ (that is contained in $S^3\setminus N$). So using a standard inner most disk argument we can use this ball to guide an isotopy to remove the non-essential intersection curves.  

We now have the homologically essential intersections of $\partial N'$ and $A_2$ to consider. Suppose two such curves cobound an annulus $A'$ in the complement of $N'$. Notice that the intersections of $A_2$ with $\partial N'$ are parallel to $K_{(p,q)}$ and thus have slope $q/p\not\in\Z$ on $\partial N'$. If $A'$ was essential we see that Lemma~\ref{annuli} implies that the boundary of $A'$ has an integral slope on $\partial N'$. Since this is not the case we see that $A'$ is not a homologically essential annulus. Thus it is boundary parallel in the complement of $N'$ ({\em cf.\ }\cite[Proposition~9.3.9]{Martelli:book}) and can be isotoped to $\partial N'$ and then into $N'$, thus removing the two curves on $\partial A'$ from $A_2\cap \partial N'$. Continuing with this we can remove all intersections between $A_2$ and $\partial N'$. Thus $A_2$ is contained in $N'$. 

Now one can easily see that $N'\setminus N$ is a Seifert fibered space over the annulus with one singular fiber (just note that there is a Seifert fibration of $N'$ with regular fibers parallel to $K_{(p,q)}$ and then remove a neighborhood of a regular fiber to get $N'\setminus N$). Now $A_1$ and $A_2$ are both essential annuli in $N'\setminus N$ that include fibers of the fibration at their boundary. Thus they must both be vertical annuli (that is unions of fibers) since any incompressible surface in a Seifert fibered space must be vertical or horizontal. They are both vertical annuli with boundary on the same boundary component of $N'\setminus N$. There are two such annuli (up to isotopy), one whose projection to the base annulus contains the projection of the singular fiber and the one that does not. The one that does not contain the projection of the singular fiber is not essential in $N'\setminus N$. So they are isotopic in $N'\setminus N$. 
\end{proof}

\section{Positive cables}\label{pcc}
In this section we prove our theorems about positive cables. We begin by putting such cables on boundaries of neighborhoods of Legendrian knots. 

\begin{lemma}\label{positive-onboundary}
If $q/p> \lceil \omega(K) \rceil$ then any Legendrian $L\in \mathcal{L}(K_{(p,q)})$ can be placed on a convex torus $T$ that bounds a solid torus whose core is in the knot type $K$. Moreover, we may assume that $T$  is the boundary of a standard neighborhood of a Legendrian $L'\in \mathcal{L}(K)$. (We note that when we do this, the characteristic foliation on $T$ might not be standard and in particular $L$ might not be a ruling curve.)
\end{lemma}

\begin{proof}
We first notice that $\tb(L)<pq$. If this were not the case then the twisting of $\xi$ with respect to a torus $T$ bounding a solid torus in the knot type of $K$ would be greater than or equal to 0, that is $\tw(L,T)\geq 0$. Thus after possibly stabilizing $L$ we can assume that it sits on $T$ with 0 contact twisting along $L$. Thus we can make $T$ convex without moving $L$. So $L$ sits on a convex torus $T$. Since $\tw(L,T)=-\frac 12 (\Gamma_T\bigcdot L)$ we see that the dividing curves of $T$ are parallel to $L$, that is the slope of $\Gamma_T$ is $q/p$. This contradicts the fact that $q/p$ was chosen larger than $ \lceil \omega(K) \rceil$ and so is larger than the width of $K$. Thus we know $\tb(L)<pq$ and hence $\tw(L,T)<0$ and we can always find a convex torus $T$ which contains $L$. 

Let $s$ be the slope of the dividing curves $\Gamma_T$ and $n$ the number of dividing curves. Let $S$ be the solid torus $T$ bounds. If $s\in \Z$ and $n=2$ then $S$ is a regular neighborhood of a Legendrian knot in $\mathcal{L}(K)$ and we are done. For the other cases, by Lemma~\ref{anyslopeless} we know that inside of $S$ there is another solid torus $S'$ that has convex boundary with 2 dividing curves of slope $m$ where $m = \lfloor s \rfloor$. We can take the ruling slope of $\partial S'$ to be $q/p$. Now consider an annulus $A$ from a ruling curve on $\partial S'$ to $L$. We claim that there are no bypasses for $\partial S'$ on $A$. To see this we consider two cases. If $q/p\geq \tbb(K)+1$ and there were such a bypass, we could attach it to $\partial S'$ to get a torus $T'$ in $S-S'$ with dividing slope $m+1$ (since there is an edge in the Farey tessellation from $m$ to $m+1$ and $m+1$ is the closest point to $q/p$ with this property, so Lemma~\ref{bypassattach} says that the bypass attachment will produce a torus of this slope). But $m\leq s<m+1$ and thus $S-S'$ is not minimally twisting. This implies that the contact structure on $S$ is overtwisted because we can realize the slope $\infty$ by a convex torus parallel to the boundary by Lemma~\ref{anyslope}, thus such a bypass does not exist. In the other case we have $ \omega(K)  <q/p< \tbb(K)+1$. Notice that in this case $\omega(K)=\tbb(K)$ since otherwise $q/p$ would need to be larger than $\lceil \omega(K) \rceil=\tbb(K)+1$. Now if $m<\tbb(K)$ then we would get the same contradiction as above. If $m=\tbb(K)$ then we must also have $s=m$ since $\omega(K)=\tbb(K)=m$. Now there are no bypasses in this case since they would increase the slope of the torus $\partial S'$ which is clearly not possible. 

We can now see that any dividing curve on $A$ that starts on a ruling curve of $\partial S'$ ends on $L$. However, there might be bypasses of $L$ on $A$. 
Now let $S''$ be $S'$ union a small $I$-invariant neighborhood of $A$. After edge-rounding, one may check that $\partial S''$ is convex with 2 dividing curves of slope $m$ and that $L$ is a part of its characteristic foliation.
\end{proof}

We notice that given any Legendrian $L$, it has a standard neighborhood $N$ with convex boundary having dividing slope $\tb(L)$. Moreover we can assume the slope of the ruling curves on $\partial N$ are any rational number not equal to $\tb(L)$. Thus we can take them to have slope $q/p$. Denote by $L_{(p,q)}$ one of these ruling curves. This will be called the \dfn{Legendrian $(p,q)$-cable of $L$}.
\begin{lemma}[Etnyre--Honda, 2004 \cite{EH04}]\label{tbandr}
If $L_{(p,q)}$ is the Legendrian $(p,q)$-cable of $L$, then 
\[
\tb(L_{(p,q)}) = pq- |p\, \tb(L)-q|
\]
and 
\[
\rot(L_{(p,q)}) = p\, \rot(L).
\]
\end{lemma}

Recall the cone of $L$ is 
\[
C(L)=\{S^k_+S^l_- (L) : \text{ for all } k \text{ and } l\}.
\]

We define the \dfn{$(p,q)$ cone of $L$} as 
\[
C_{(p,q)} (L)= C(L_{(p,q)}).
\]
\begin{remark}
Since stabilization is well-defined we see that two elements in $C(L)$ are isotopic if and only if they have the same rotation numbers and Thurston-Bennequin invariants. 
\end{remark}
\begin{lemma}\label{cones}
The $(p,q)$ cone of $L$ is the union of the $(p,q)$ diamonds for all $L'$ in the cone of $L$. That is
\[
C_{(p,q)}(L)= \bigcup_{L'\in C(L)} D_{(p,q)}(L').
\]
Moreover, if $L'$ and $L''$ are distinct in $C(L)$ then the diamonds $D_{(p,q)}(L')$ and $D_{(p,q)}(L'')$ are disjoint.
\end{lemma}
\begin{proof}
Suppose that $L'$ is an element of $C(L)$, that is $L'$ is a stabilization of $L$. If $N$ is a standard neighborhood of $L$, then we can stabilize $L$ inside of $N$ to obtain $L'$ and then take a standard neighborhood $N'$ of $L'$ inside of $N$. We now have by definition that $L_{(p,q)}$ and $L'_{(p,q)}$ are ruling curves on the boundary of $N$ and $N'$ respectively. Let $A$ be a convex annulus connecting them in $N\setminus N'$. Notice that there are no bypasses for $\partial N$ on $A$ because if there were we could attach a bypass to $\partial N$ to get a convex torus $T$ in $N-N'$ with boundary slope $\infty$, thus contradicting the tightness of the contact structure on $N$. 
 
The imbalance principle shows that there will be bypasses for $L'_{(p,q)}$ along $A$.  We can use this to destabilize $L'_{(p,q)}$. If we keep destabilizing using bypasses on $A$ until there there are no more, then the destabilized $L'_{(p,q)}$ co-bounds a sub-annulus $A'$ of $A$ with $L_{(p,q)}$. Now the dividing curves on $A'$ all run from one boundary component to the other (since there are no boundary parallel dividing curves for either boundary component of $A'$). Thus we can foliate $A'$ by ``ruling curves". And these ruling curves give a Legendrian isotopy from our destabilized $L'_{(p,q)}$ and $L_{(p,q)}$. That is $L'_{(p,q)}$ is in $C_{(p,q)}(L)$. 

Since $D_{(p,q)}(L')$ is defined to be some stabilizations of $L'_{(p,q)}$ they will also be stabilizations of $L_{(p,q)}$. In other words $D_{(p,q)}(L')\subset C_{(p,q)}(L)$ and we have
\[
  \bigcup_{L' \in C(L)}D_{(p,q)}(L') \subset C_{(p,q)}(L).  
\] 
From the formula in Lemma~\ref{tbandr} it is clear that all the diamonds $D_{(p,q)}(L')$ are disjoint and it is not hard to check that $S^k_+S^l_- (L_{(p,q)}) \in D_{(p,q)}(S^{\lfloor k/p \rfloor}_+S^{\lfloor l/p \rfloor}_-(L))$. Thus we obtain
\[
  C_{(p,q)}(L) \subset \bigcup_{L' \in C(L)}D_{(p,q)}(L')
\]
to complete the proof. 
 \end{proof}

\begin{lemma}\label{mainlem}
If $q/p>\lceil \omega(K) \rceil$ is not an integer and $L\in \mathcal{L}(K_{(p,q)})$, then there is a unique $L'\in \mathcal{L}(K)$ with $L\in D_{(p,q)}(L')$ 
\end{lemma}
\begin{remark}
If $q/p$ is an integer, then $K_{(p,q)}$ is isotopic to $K$ so the excluded cases from this lemma are trivial. 
\end{remark}
\begin{remark}
Notice that this lemma says that the knot $L'$ is an invariant of $L$. We call this the \dfn{underlying knot} of $L$ and denote it by $\under(L)$. 
\end{remark}
\begin{proof}
From Lemma~\ref{positive-onboundary} we know that $L$ sits on the boundary of some solid torus $N$ that is a standard neighborhood of some $\overline{L}\in \mathcal{L}(K)$. If $L$ is not a ruling curve on $\partial N$ then $L$ does not intersect the dividing curves of $\partial N$ minimally and we can use the bypass that results to destabilize $L$ to another Legendrian on $\partial N$. We may continue to do this until $L$ is a ruling curve on $\partial N$. That is, $L$ destabilizes to $\overline{L}_{(p,q)}$ and in particular is in the $(p,q)$ cone $C_{(p,q)}(\overline{L})$. Now from Lemma~\ref{cones} we see that there is a unique $L'\in C(\overline{L})$ such that $L\in D_{(p,q)}(L')$. 

We are left to show that $L'$ is uniquely determined ($L'$ is unique among Legendrian knots in $C(\overline{L})$ but we need to see it is the unique Legendrian knot in $\mathcal{L}(K)$). Suppose that $L$ sits on the boundary of a standard neighborhood $N'$ of $L'$ and on the boundary of a standard neighborhood of $N''$ of $L''$ so that $L'$ and $L''$ both have the same rotation number and Thurston-Bennequin invariant (so that $L$ is in the $(p,q)$ diamond of both $L'$ and $L''$). We show that $L'$ and $L''$ are Legendrian isotopic, and thus the knot underlying $L$ is well-defined. 

Let $N$ be a small neighborhood of $L$. We claim that we can isotope $T'=\partial N'$ fixing $L$ to agree with $T''=\partial N''$ in $N$. To do so, take a small I-invariant neighborhood $B'$ of $T'$ and smoothly isotope $T'$ in this I-invariant neighborhood fixing $L$ to a torus $\widetilde{T'}$ that agrees with $T''$ in $N$. Now $\widetilde{T'}$ is a convex torus in $B'$ and so has the same dividing slope as $T'$ but might have more than two dividing curves. We will show that we can further isotope $\widetilde{T'}$ so that it has only two dividing curves and still agrees with $T''$ in $N$. 

If $\widetilde{T'}$ already has only two dividing curves then we are done. So we suppose it does not. We notice that we could arrange the characteristic foliation on $\partial B'$ so that it is standard with ruling curves of the same slope as $L$. Let $\widetilde{C}$ be a Legendrian curve in $\widetilde{T'}$ that is disjoint from $L$. Take an annulus $A$ in $B'$ between a ruling curve on one boundary component of $B'$ and $\widetilde{C}$ on $\widetilde{T'}$. Since the ruling curves on $\partial B'$ intersect the dividing curves on $\partial B'$ minimally we see that there is a bypass for $\widetilde{T'}$ along $A$ that is disjoint from $L$. This bypass will either reduce the number of dividing curves of $\widetilde{T'}$, increase the number, or do nothing (it cannot change the slope of the dividing curves since we are in an $I$-invariant neighborhood and we are assuming there are more than $2$ dividing curves). Notice that we can continue to attach bypasses to $\widetilde{T'}$ along $A$ until the number of dividing curves is reduced to $2$ (we note that the number will eventually have to reduce to $2$ since $A$ intersects the dividing curves on $\partial B'$ minimally and there are two dividing curves there of the same slope as on $\widetilde{T'}$). This will result in the desired torus, which we rename $T'$ for convenience. 

Now denote two annuli $T' \setminus N$ and $T'' \setminus N$ by $A'$ and $A''$ respectively. Lemma~\ref{smooth-isotopy} says that $A'$ is smoothly isotopic, rel boundary, to $A''$ in $S^3\setminus N$. 
Given the topological isotopy from $A'$ to $A''$ we can use Colin's isotopy discretization \cite[Lemma~3.10]{Honda02} to find a sequence of annuli $A_0=A', A_1, \ldots, A_k=A''$ such that $A_{i+1}$ is obtained from $A_i$ by attaching a bypass from either the ``inside" or ``outside". Note the $A_i$ together with $T'\cap N$ form a torus $T_i$ that bounds a solid torus $N_i$; the bypass is attached from the inside if the bypass is inside of $N_i$ and attached from the outside otherwise. 

We will inductively show that each $N_i$ contains a standard neighborhood of some Legendrian knot $L_i$ that is either $L'$ or stabilizes to $L'$ (in other words $L'$ destabilizes to $L_i$). Thus $N_k$ will be a standard neighborhood of $L'$ and hence $L'$ and $L''$ will be isotopic.

We first show that the slope of dividing curves of $N_i$ is always in $[\tb(L'), \infty)$ for any $i$. This is true for $N_1$, so inductively assume it is true for $N_i$. Suppose the slope of $N_{i+1}$ is in $[-\infty, \tb(L')).$ If the slope were less than $\tb(L')-1$ then we claim that $L$ would have to intersect the dividing curves of $T_{i+1}$ more times than it intersects the dividing curves of $T_i$, but since the bypass was attached in the complement of $L$ this is not possible.

To verify the claim notice that since $L$ is in the diamond $D_{(p,q)}(L')$ it is at most a $2p-2$ stabilization of $L'_{(p,q)}$. That is, the contact twisting of $L$ relative to $T_{i+1}$ is greater than or equal to 
\[
  -|\tb(L') \bigcdot q/p|-2p+2,  
\] 
and hence the intersection of $L$ with one of the dividing curves of $T_{i+1}$ is less than or equal to $|p\tb(L')-q|+2p-2$. However, we see that 
\[
  |(\tb(L')-2) \bigcdot q/p| = |p\tb(L') - q| + 2p,
\]
so the curve of slope $\tb(L') - 2$ intersects the slope $q/p$ too many times and hence by Lemma~\ref{intersections}, all slopes less than $n-1$ intersect $q/p$ too many times.

If the slope of $T_{i+1}$ is in $[\tb(L')-1, \tb(L'))$, then $N_{i+1}$ contains a standard neighborhood $\widehat{N}$ of some Legendrian knot $\widehat{L}$ that is a stabilization of $L'$. Take an annulus $A$ between $\partial N_{i+1}$ and $\partial \widehat{N}$ with one boundary component being $L$ and the other being a $q/p$ ruling curve on $\partial \widehat{N}$. There can be no bypasses for $\widehat{N}$ along $A$ as attaching one would create a torus of slope $\tb(L')$ contained in $N_{i+1}$, which is not possible. Thus we can destabilize $L$ until it is a ruling curve on $\widehat{N}$. Thus $L$ would have to be in the $(p,q)$ cone of $\widehat{L}$, but this cone is disjoint from $D_{(p,q)}(L')$, and thus the slope of $T_{i+1}$ cannot be in $[\tb(L')-1, \tb(L'))$. Hence the slope of $N_{i+1}$ cannot be in $[-\infty, \tb(L'))$ as claimed. 

Now each $N_i$ contains a standard neighborhood of a Legendrian knot $L_i$ in $\mathcal{L}(K)$ and has dividing curves with slope in $[\tb(L_i),\tb(L_i)+1)$ (possibly with more than 2 dividing curves). We inductively assume that $L_i$ is $L'$ or stabilizes to $L'$. If the bypass is attached from the outside, then there are three possible cases: it does nothing, it changes the number of dividing curves, or it increases the dividing slope. For the first two cases, $L_{i+1}$ is $L_i$. In the last case $L_{i+1}$ is $L_i$ or a destabilization of $L_i$, but in either case $L_{i+1}$ stabilizes to $L'$. If the bypass is attached from the inside then there are again three cases. The first two are the same as before and $L_{i+1}=L_i$. In the last case the slope of the dividing curves on $N_{i+1}$ decreases. If that slope is still in $[\tb(L_i),\tb(L_i)+1)$ then $L_{i+1}=L_i$. If not, then $L_{i+1}$ is a stabilization of $L_i$. If $L_{i+1}$ does not stabilize to $L'$ then the $(p,q)$ cone of $L_{i+1}$ would not contain the $(p,q)$ diamond $D_{(p,q)}(L')$ and thus $L$ could not sit on a convex torus neighborhood of $L_{i+1}$. Therefore, $L_{i+1}$ must stabilize to $L'$, thus completing our argument.
\end{proof}

\begin{proof}[Proof of Theorem~\ref{classification-positive}]
Lemmas~\ref{cones} and~\ref{mainlem} show that if $q/p>\lceil \omega(K) \rceil$ is not an integer, then two knots $L_0, L_1\in \mathcal{L}(K_{(p,q)})$ are Legendrian isotopic if and only if they have the same rotation numbers, Thurston-Bennequin invariants, and underlying knots. 
\end{proof}

\begin{proof}[Proof of Theorem~\ref{postransverse}]
Recall that by Theorem~\ref{traniso} the transverse push-offs of two Legendrian knots $L$ and $L'$ are transversely isotopic if and only if $L$ and $L'$ have a common negative stabilization \cite{EtnyreHonda01b}. Given this, the result clearly follows from the classification of Legendrian knots in positive cables, Theorem~\ref{classification-positive}.
\end{proof}

\section{Legendrian large cables}\label{llcsec}
The proof of Theorem~\ref{llc-model} can be thought of as a generalization of the proof of \cite[Theorem~1.2]{EH04}. Recall that Theorem~\ref{llc-model} says that any Legendrian large $(p,q)$-cable is contained in a balanced continued fraction block with center slope $q/p$ and any such continued fraction block gives a Legendrian large cable. 

\begin{proof}[Proof of Theorem~\ref{llc-model}]
We start with Item~(\ref{llc-model:existence}). Let $L \in \mathcal{L}(K_{(p,q)})$ be a Legendrian large cable with $\tb(L) = pq + m$ and $N$ a standard neighborhood of $L$. Take an essential annulus $A$ in $\overline{Y \setminus \partial N}$ so that $\partial N \setminus A$ is a disjoint union of two annuli $A_0, A_1$ and for $i=0,1$, $A \cup A_i$ is a torus that bounds a solid torus whose core is in the knot type $K$. We can assume $\partial N$ is convex with ruling curves of slope $pq$. We can now take $A$ to have Legendrian boundary being two of these ruling curves. Perturb $A$ to be convex and take a small I-invariant neighborhood $A \times [0,1]$ (that intersects $\partial N$ in annuli foliated by ruling curves). Then $N \cup A \times [0,1]$ is a thickened torus $T^2 \times [0,1]$. 

Now consider the dividing curves on $A$. If there is a boundary parallel dividing curve on $A$, then this gives a bypass for $N$ with slope $pq$. Since $\tb(L)$ is bigger than $pq$, attaching this bypass results in dividing curves with meridional slope so the contact structure on the solid torus bounded by this torus is overtwisted. This contradicts the fact that $L$ lives in a tight contact manifold. Thus the dividing curves on $A$ consist of $2m$ properly embedded non-separating arcs. Let $T_i = T^2 \times \{i\}$ for $i=0,1$. Now change the coordinates on $T_i$ so that the slope of the boundary of $A$ becomes $0$. The dividing curves on $A$ may have holonomy, that is a non-zero slope in this coordinate system, see Figure~\ref{holonomy}. 
    
\begin{figure}[htb]{\tiny
\begin{overpic}
{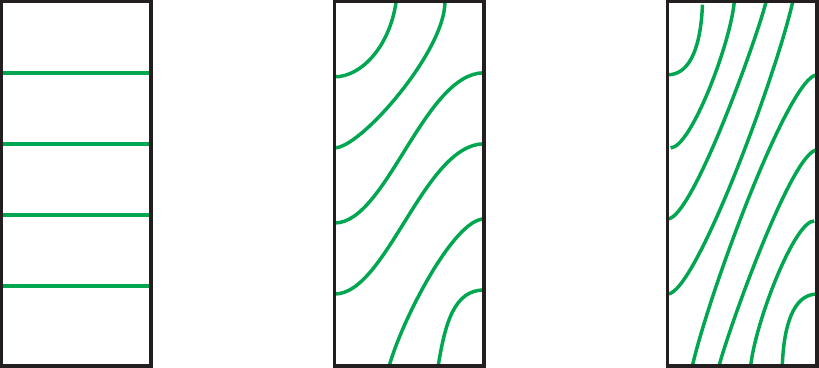}
\end{overpic}}
\caption{Possible holonomy on the annulus $A$. The annuli have holonomy $0$, $1/2$, and $1$, respectively.}
\label{holonomy}
\end{figure}

We can modify the holonomy by $\pm 1$ by changing the coordinate of $T^2$ again. Thus we can assume that there exists only $k/m$-holonomy for $0 \leq k < m$. Then after edge rounding, one may check that we obtain $T^2 \times [0,1]$ with dividing slope $s_0= m/(-1+k)$ on $T_0$ and $s_1=m/(1+k)$ on $T_1$. This follows since a $0$ sloped curve will intersect the dividing curves on $T_i$, $2m$ times and an $\infty$ sloped curve will intersect the dividing curves on $T_0$, $-2+2k$ and on $T_1$, $2+2k$ times. 

Below we will see that $k$ must be $0$. Assuming that for the moment, we complete the proof that a neighborhood of $K$ is contained in length $2m$ balanced continued fraction blocks of which the center slope is $q/p$ (recall in the coordinates chosen on $T$ above, $q/p$ is the slope $0$). In this case, we have that the contact structure on $T^2 \times [0,1]$ has boundary slopes $s_0=-m$ and $s_1=m$. 
Let $C$ be the core curve of $A$ and notice that the relative Euler class evaluated on $C\times [0,1]$ is 0 since the contact structure on $A\times [0,1]$ is $[0,1]$-invariant. We can factor $T \times [0,1]$ into $2m$ basic slices $B_i$ where $B_i$ has boundary slopes $-m+i$ and $-m+i+1$ for $i=0, \ldots, 2m-1$. The possible relative Euler classes evaluated on $C\times [0,1]$ can be $[-m,m] \cap \Z$ and they are the difference between the number of positive and negative basic slices. Thus the number of positive and negative basic slices are the same and we have that $T\times [0,1]$ is a balanced continued fraction block.

We will now show that $k$ must be $0$. Suppose not, then we must have $m>1$ and $0<k<m$. Stabilize $L$ to be $L'$ with $\tb(L') = pq+1$. Take a standard neighborhood of $L'$ inside of $N$ and denote it by $N'$. Also take a convex annulus $A'$ in $T \times [0,1]$ which is smoothly isotopic to $A$ and whose boundary consists of ruling curves on $\partial N'$. Take a small I-invariant neighborhood of $A'$ and by gluing this to $N'$ as before, we obtain a thickened torus $T^2 \times I$ contained in $T \times [0,1]$ for $I \subset [0,1]$. After edge rounding, we obtain balanced continued fraction blocks $B=T^2 \times I$ with boundary slopes $1/(-1+l)$ and $1/(1+l))$ according to the argument above. Notice that since we cannot change coordinates on $T^2$ as we did above, there is no restriction on $l$. Recall we can realize any slope in $[1/(-1+l), 1/(1+l)]$ by a convex torus parallel to the boundary of $B$. In particular, notice that we can always find a convex torus in $B$ with dividing slope $m/k$ if $l\not= 0,1$. Since $B$ is contained in $T \times [0,1]$ which has boundary slopes $m/(-1+k)$ and $m/(1+k)$ and the interval $[m/(-1+k), m/(1+k)]$ does not contain $m/k$, we see that $T\times [0,1]$ is not minimally twisting unless $l=0$ or $1$. If $T\times [0,1]$ is not minimally twisting then the ambient contact manifold is not tight, thus we must have $l=0$ or $1$. 
Notice that when $m>1$, the interval of slopes $[m/(-1+k), m/(1+k)]$ contains $[-\infty, 1]$ and we also see that $[1/(-1+l), 1/(1+l)]$ is contained in $[-\infty, 1]$. Thus there is $J\subset [0,1]$ that contains $I$ such that $T\times J$ has boundary slopes $-\infty$ and $1$ and contains $B$. But $T\times J$ is a basic slice, which cannot contain continued fraction blocks with mixed signs unless it is overtwisted. Thus we must have $k=0$. 

To prove the converse, we first observe that for every $m$ there is a length $2m$ balanced continued fraction block $B_m$ that decomposes as described above. For example, take the $(m+1,-1)$-cable $L$ of the unknot shown in Figure~\ref{lc} for $m=1$. 
    
\begin{figure}[htb]{\tiny
\begin{overpic}
{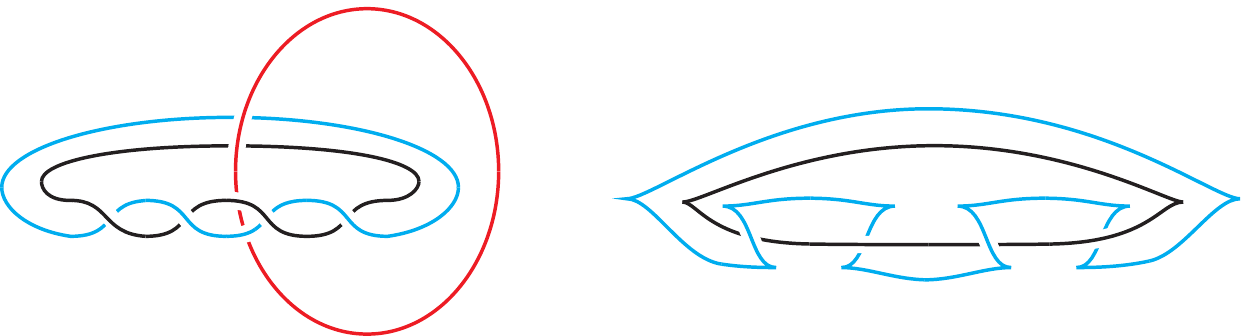}
\end{overpic}}
\caption{A $(2,-1)$-cable of the unknot. On the left we see the blue and red curves form a Hopf link and the black curve is a $(-1,2)$-cable of the blue curve but a $(2,-1)$-cable of the red curve. On the right a Legendrian realization of the black curve, that is the $(2,-1)$-cable of the red curve, with Thurston-Bennequin invariant $-1$ and it is sitting in the complement of realization of the blue unknot that has been stabilized twice positively and twice negatively. That is, we see that there is a balanced continued fractions block of length four in the neighborhood of the red unknot with central slope $-2$.}
\label{lc}
\end{figure}

It clearly has Thurston-Bennequin invariant $-1$ and so sitting on the cable torus its twisting relative to the torus is $m$. Now the argument above says we can take a neighborhood $N$ of $L$ and let $A$ be the part of the cable torus outside of $N$. We can make $A$ convex and as above will have dividing curves that run from one boundary to the other. So we can let $B_m$ be a neighborhood of $A$ together with $N$. The argument above says that this is a balanced continued fraction block with center slope $-1/(m+1)$ and using a diffeomorphism of $B_m$ we can arrange any center slope we like. 

Now assume that the knot $K$ has a neighborhood that contains a balanced continued fraction block $B$ with length $2m$ whose center slope is $q/p$. Observe that there is a contactomorphism from $B_m$ to $B$. But inside $B_m$ we have a cable with twisting $m$ larger than the torus framing. So its image $L$ in $B$ will be $(p,q)$-cable of $K$ with $\tb(L)=pq+m$. This completes the proof Item~(\ref{llc-model:existence}).

For Item~(\ref{llc-model:unique}), let $B$ be a length $2m$ balanced continued fraction block with the center slope $q/p$ containing $L$. Let $N$ be a standard neighborhood of $L$ and $A$ an essential annulus in $\overline{B \setminus N}$ whose boundary components are Legendrian ruling curves of $\partial N$. In the proof of Item~(\ref{llc-model:existence}), we showed that the union of $N$ and an $I$-invariant neighborhood of $A$ is a length $2m$ balanced continued fraction block. Since it is contained in $B$, they are clearly contact isotopic. Now suppose that there exist two length $2m$ balanced continued fraction blocks $B$ and $B'$ containing $L$. As discussed above, we can consider $B$ and $B'$ as a union of $N$ and an $I$-invariant neighborhood of essential annuli $A$ and $A'$, respectively. By Lemma~\ref{smooth-isotopy}, there is a smooth isotopy from $A$ to $A'$ fixing $N$. By isotopy discretization \cite[Lemma~3.10]{Honda02}, there is a sequence of annuli related by bypass attachments that go from $A$ to $A'$. These bypasses can have one of three effects on the dividing curves on the annulus. 

First, the bypass might produces a boundary-parallel dividing curve on the annulus, which gives a bypass on $N$ with slope $pq$. After attaching the bypass, $N$ thickens to a solid torus that contains a dividing curve with meridional slope, which implies overtwistedness. 

Second, the bypass might change the slope of the diving curves. This only happens when the number of dividing curves on $A$ is $2$ and hence the length of the continued fraction block is $2$. Then attaching the bypass increases or decreases the holonomy by $1$. However, we will see as in the proof of Item~(\ref{llc-model:existence}), changing the holonomy produces an overtwisted contact structure. Indeed, suppose a bypass is attached to $A$ from the front to obtain an annulus $A'$ and the holonomy went from $l$ to $l+1$ (here we have chosen coordinates on the torus as in the proof of Item~(\ref{llc-model:existence})). Then the $N$ union the neighborhood of $A$ gives a thickened torus $T\times I$ with boundary having dividing slopes $1/(-1+l)$ and $1/(1+l)$. Doing the same for $A'$ yields a thickened torus $T'\times I$ with boundary having dividing slopes $1/l$ and $1/(2+l)$. Since the bypass was attached from the front of $A$, the back face of $T\times I$ and the front face of $T'\times I$ are disjoint and cobound a thickened torus whose boundary has dividing curves of slope $1/(-1+l)$ and $1/(2+l)$ that contains $T'\times I$. In particular, it contains a torus of slope $1/l$.  But as $1/l$ is not in the interval $[1/(-1+l), 1/(2+l)]$ (recall this is the interval on the boundary of the Farey graph that begins at $1/(-1+l)$ and goes clockwise until it gets to $1/(2+l)$) and thus the region bounded by the back face of $T\times I$ and front face of $T'\times I$ is not minimally twisting, but is contained in a solid torus and thus the contact structure on this torus must be overtwisted. 

Thus we see the only possible bypasses that can be attached to $A$ are trivial ones and this implies that there is a contact isotopy taking $B$ to $B'$. This completes the proof of Item~(\ref{llc-model:unique}).
\end{proof}

We now turn to the proof of Lemma~\ref{llcrot} that shows how to compute the rotation number of a Legendrian large cable from a ruling curve on the boundary of the associated continued fraction block. 

\begin{proof}[Proof of Lemma~\ref{llcrot}]
It is well-known that if $N$ is a standard neighborhood of a Legendrian knot $L$ with longitudinal ruling curves (that is having any integer slope different from $\tb(L)$) then the rotation number of $L$ and a ruling curve (oriented in the same direction as $L$) are the same. To see this notice that one can trivialize the contact planes on $N$ by taking a tangent vector field to $L$ and extending it to all of $N$. Now the tangent vectors to the ruling curve and to $L$ will not rotate with respect to this trivialization, this implies they have the same rotation number. 

The continued fraction block $B$ associated to $L$ is obtained from $N$ by attaching an $I$-invariant neighborhood of an annulus with boundary $pq$-ruling curves on $\partial N$. So the ruling curves on $\partial N$ can also be taken to be ruling curves on $\partial B$. In the coordinates on $\partial B$ coming from the underlying knot that is being cabled, these ruling curves will have slope $q/p$. Thus the rotation number of such curves agrees with the rotation number of $L$.  
\end{proof}

Recall Theorem~\ref{llc-isotopy} says that a Legendrian large cable is uniquely determined by and determines the continued fraction block associated to it by Theorem~\ref{llc-model}.
\begin{proof}[Proof of Theorem~\ref{llc-isotopy}]
Suppose $L$ and $L'$ are Legendrian isotopic. Denote their standard neighborhoods by $N$ and $N'$, respectively and let $B$ and $B'$ be the balanced continued fraction blocks associated to $L$ and $L'$, respectively. Since $L$ and $L'$ are isotopic, there is an ambient contact isotopy $\phi_t$ of $(S^3, \xi_{std})$ taking $N$ to $N'$. Relabel $\phi_1(B)$ as $B$. Then $B$ and $B'$ agree on $N'$ and they are contact isotopic by Item~(\ref{llc-model:unique}) of Theorem~\ref{llc-model}. 
  
Conversely, suppose $B$ and $B'$ are contact isotopic. By contact isotopy extension theorem ({\em cf.\ }\cite{Geiges:book}), there is an ambient contact isotopy $\phi_t:(S^3,\xi_{std}) \rightarrow (S^3,\xi_{std})$ taking $B$ to $B'$. Thus, $S^3 \setminus B$ and $S^3 \setminus B'$ are contactomorphic. Since both $B \setminus N$ and $B' \setminus N'$ are I-invariant neighborhoods of convex annuli, $(S^3 \setminus B) \cup (B \setminus N)$ and $(S^3 \setminus B') \cup (B' \setminus N')$ are equivalent to gluing neighborhood of a convex annulus to $(S^3 \setminus B)$ and $(S^3 \setminus B')$, respectively. Thus, $S^3 \setminus N$ and $S^3 \setminus N'$ are contactomorphic and this implies that $L$ and $L'$ are Legendrian isotopic by Lemma~\ref{Legendrian-isotopy}.
\end{proof}

We now establish the relation between a Legendrian large cables and its stabilizations. 
\begin{proof}[Proof of Theorem~\ref{stabilizellc}]
Let $N$ be the standard neighborhood of $L$. Recall from the proof of Theorem~\ref{llc-model} the continued fraction block $B$ of length $2m$ is obtained from $N$ by adding an $I$-invariant neighborhood of an annulus $A$ with boundary $pq$ sloped ruling curves on $\partial N$. Now inside $N$ we can stabilize $L$ to get $S_\pm(L)$ with a neighborhood $N_\pm\subset N$. We can extend $A$ to an annulus $A_\pm$ with boundary on $\partial N_\pm$. We can now form a continued fraction block $B_\pm$ from $N_\pm$ by adding an invariant neighborhood of $A_\pm$ so that $B_\pm\subset B$. From the proof of Theorem~\ref{llc-model} we see that $B_\pm$ has length $2(m-1)$. Since it is centered about $q/p$ we see that $B\setminus B_\pm$ is the union of two basic slices of opposite sign as claimed. Since the $B_\pm$ are different and there are only two ways to remove outermost basic slices and still have a balanced continued fraction block, we see that one way corresponds to $B_+$ and the other to $B_-$. Which one corresponds to the positive stabilization and which to the negative one can easily be computed as follows. Let $T^2\times I$ be the basic slice in $B\setminus B_\pm$ that contains the front face of $B$. Let $A'$ be an annulus of slope $q/p$ in the basic slice with boundary a ruling curve on each boundary component of $T^2\times I$.
One may compute that the relative Euler class of the contact structure evaluated on $A'$ is $\pm 1$. This is also the difference between the rotation numbers of the ruling curves on the front and back face of $T^2\times I$. Since by Lemma~\ref{llcrot} we know this corresponds to the rotation number of $S_\pm(L)$, we are done.  

Now given $L$ contained in $B$, stabilize the knot $m$ times to get $L'$. We know $\tb(L')=pq$ so we may put it on a convex torus $T$ inside of $B$ as one of the Legendrian divides. As argued below in the proof of Lemma~\ref{negative-onboundary} we may assume that $T$ has only two dividing curves. Thus there is a contactomorphism of $B$ that takes $T$ to one of the $T_i$, and hence a contactomorphism of $S^3$ taking $T$ to $T_i$. So we see that up to contactomorphism $L'$ is a Legendrian divide on $T_i$, and by  Lemma~\ref{Legendrian-isotopy} there will be a Legendrian isotopy from $L'$ to this divide. 
\end{proof}

We now move to find an upper bound on the Thurston-Bennequin invariant of any $(p,q)$-cable. 
\begin{proof}[Proof of Theorem~\ref{tbbound}]
Suppose that $q/p\leq\lceil\omega(K)\rceil$ is not an integer.
By Theorem~\ref{llc-model} we know if there exists $L\in \mathcal{L}(K_{(p,q)})$ with $\tb(L)=pq+m$ then there is a balanced continued fraction block of length $2m$ with center slope $q/p$. That means that there is a solid torus with boundary slope $q/p$ that contains half of this continued fraction block. In particular it will have to be part of the tail of the path from $\lfloor q/p \rfloor$ to $q/p$ in the Farey graph. So $m$ must be less than or equal to the tail. 

When $q/p>\lceil\omega(K)\rceil$ is not an integer, we get the bound from Theorem~\ref{classification-positive}.
\end{proof}

We end this section by showing that the width of Yasui's example $K_m$ discussed above has contact width at least  $-\frac{1}{2\left\lfloor \frac{3-m}{4} \right\rfloor - 1}$.
\begin{proof}[Proof of Theorem~\ref{UTP-Yasui}]
  According to \cite[Proposition~4.2]{Yasui16}, $K_m$ contains a Legendrian large cable for $m \leq -5$, more precisely, $\tbb((K_m)_{(-n,1)}) = -1$ for $n \leq \left\lfloor \frac{3-m}{4} \right\rfloor$. This implies that there are length $2(\left\lfloor \frac{3-m}{4} \right\rfloor - 1)$ balanced continued fraction blocks with center slope $-\frac{1}{\left\lfloor \frac{3-m}{4} \right\rfloor}$. Thus the dividing slope of the neighborhood of $K_m$ containing the continued fraction blocks is $-\frac{1}{2\left\lfloor \frac{3-m}{4} \right\rfloor - 1}$. 
\end{proof}

\section{Negative cables}\label{negsec}
In this section we prove Theorems~\ref{classification-negative} and~\ref{classification-in-between} concerning the classification of $(p,q)$-cables of $K$ with $q/p<\omega(K)$ and $q/p\in[\omega(K), \tb(K)+1)$, respectively. We begin with a series of lemmas that establish most of the claims in those theorems. 

\begin{lemma}\label{negative-onboundary}
  Assume $K$ is a $q/p$-minimally thickenable knot. If a Legendrian knot $L \in \mathcal{L}(K_{(p,q)})$ is not a Legendrian large cable, then it can be placed on a convex torus $T$ with two dividing curves that bounds a solid torus whose core is in the knot type $K$. (As in Lemma~\ref{positive-onboundary}, the characteristic foliation on $T$ might not be standard)
\end{lemma}
\begin{proof}
  If $L$ is not a Legendrian large cable, then $\tb(L) \leq pq$. Thus any torus $T$ on which $L$ sits can be made  convex without moving $L$ and as in Lemma~\ref{positive-onboundary} if $\tw(L, T) < 0$, then we can assume $T$ has two dividing curves. Suppose $\tw(L, T) = 0$ and $T$ has $2n$ dividing curves for $n > 1$. Let $N$ be a solid torus which is bounded by $T$. Then inside of $N$, there exist a solid torus $N'$ that has convex boundary $T'$ with $2$ dividing curves of the same dividing slope of $T$. Since $K$ is $q/p$-minimally thickenable, there exists another solid torus $N''$ containing $N$ that also has convex boundary $T''$ with $2$ dividing curves of the same dividing slope of $T$. Then the contact structure restricted on $N'' \setminus N' \cong T^2 \times [-1,1]$ is contactomorphic to an I-invariant neighborhood of $T'$. Let $\alpha$ be a simple closed Legendrian curve on $T$ with slope $s/r$ such that $|q/p \bigcdot s/r| = 1$. Take a smooth annulus $A = \alpha \times [-1,1]$. Perturb $\alpha \times \{i\}$ for $i = -1, 1$ so that they become Legendrian and intersect the dividing curves on $T'$ and $T''$ exactly twice on each boundary component and perturb the annulus to be convex. Since there are more than $2$ dividing curves on $T$, $\tw(\alpha, T) < -1$, which implies that $T$ does not intersect the dividing curves on $A$ minimally. We can perturb $T$ fixing $L$ so that it intersects dividing curves exactly twice. See Figure~\ref{decreasing-intersections} for example. After $C^\infty$ perturbation, $T$ becomes a convex torus with $2$ dividing curves and $L$ is a Legendrian divide on $T$.
\end{proof}

\begin{figure}[htb]{\tiny
\begin{overpic}
{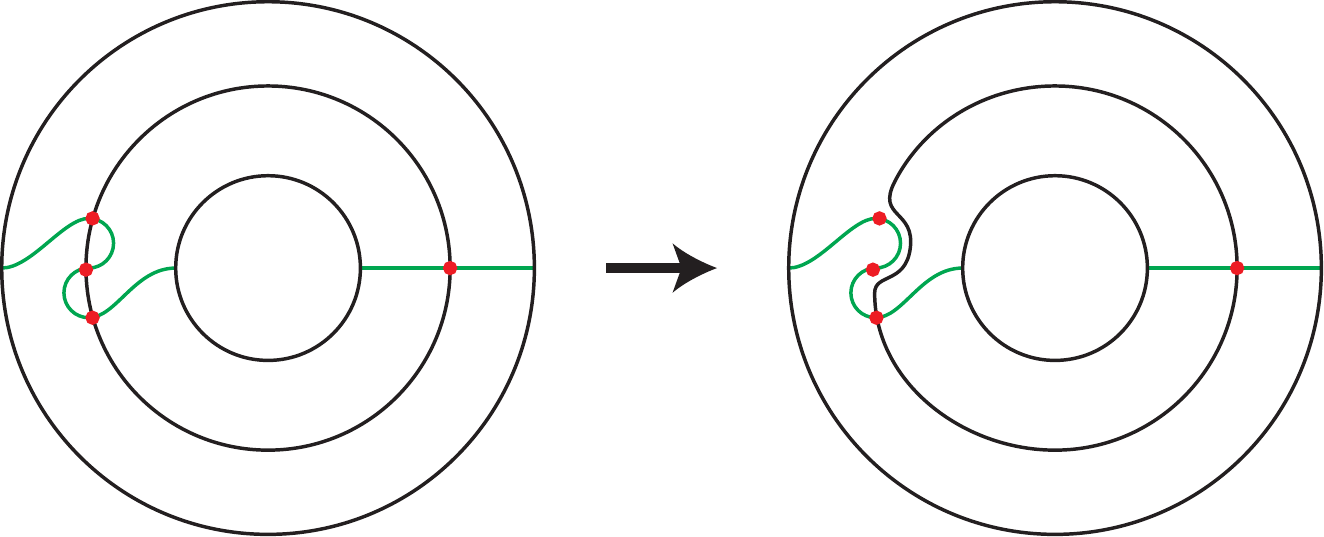}
\put(130, 141){$T'$}
\put(113, 122){$T$}
\put(98, 103){$T''$}
\put(360, 141){$T'$}
\put(340, 122){$T$}
\put(324, 103){$T''$}
\end{overpic}}
\caption{An annulus $A$ in $N'' \setminus N'$. The red dotted curves represent dividing curves.}
\label{decreasing-intersections}
\end{figure}
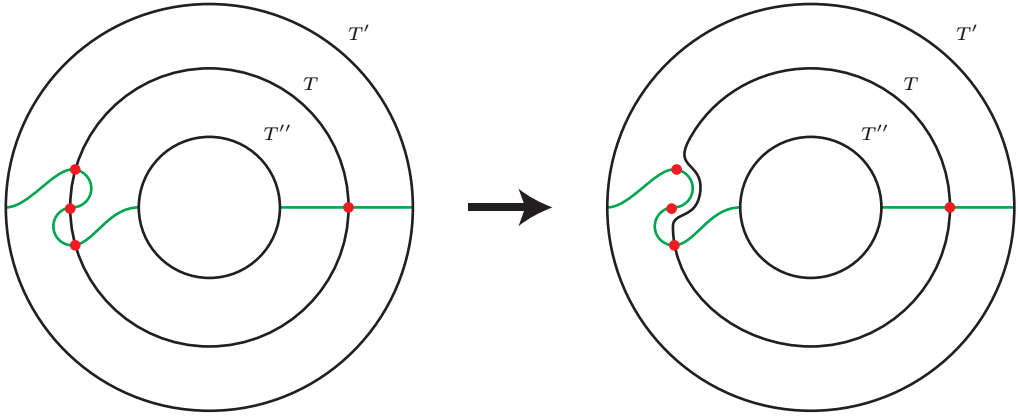

Recall from the start of Section~\ref{negcablesec} that we have slopes $a_i$ for $i\leq n+k$ and distinct (up to contact isotopy) solid torus $N_i^j$, for $j=1,\ldots, d_i$, in the knot type $K$ with convex boundary having $2$ dividing curves of  slopes $a_i$ that do not thicken to solid tori of slope $a_{i+1}$ for $i<n$. For $i=n$ we have $d_n$ distinct solid tori $N_n^j$ in the knot type $K$ with convex boundary having two dividing curves of slope $a_n$, and for $i>n$, $d_i$ distinct solid tori $N_i^j$ in the knot type $K$ that contain a balanced continued fraction block of length $2(i-n)$ with center slope of $q/p$ that do not thicken to solid tori containing a balanced continued fraction block of length $2(i+1-n)$ with center slope $q/p$. 

We also defined standard $(p,q)$-cables $L_i^j$ of $K$ to be a $q/p$ sloped ruling curve on $\partial N_i^j$ for $i<n$, a Legendrian divide on $\partial N_i^j$ for $i=n$, and the Legendrian large cable associated to the balanced continued fraction block in $N_i^j$ for $i>n$.

\begin{lemma}\label{negative-tbrot}
Let $T$ be a convex torus with two dividing curves of slope $s$ that bounds a solid torus $N$ whose core is in the knot type $K$. Suppose
$s\leq q/p$   
and $L \in \mathcal{L}(K_{(p,q)})$ is a Legendrian ruling curve on $T$ if $s<q/p$ or a Legendrian divide on $T$ if $s=q/p$. 
Factor $N$ into $N' \cup B_1 \cup ... \cup B_k$ where $N'$ is a standard neighborhood of $L' \in \mathcal{L}(K)$ with $\tb(L') = \lfloor s \rfloor$, and $B_i$ is a basic slice with dividing slopes $s_{i-1}$ and $s_i$ where $\{s_0, ..., s_k\}$ is the shortest path on Farey graph from $s_0 = \lfloor s \rfloor$ to $s_k = s$. Denote the sign of the basic slice $B_i$ by $\epsilon_i$. Then,
  \[
    \tb(L) = pq - |q/p \bigcdot s|
  \]
  and
  \[  
    \rot(L) = p \rot(L') - \sum_i{\epsilon_i} |(s_{i}\ominus s_{i-1}) \bigcdot q/p|.
  \]
\end{lemma}
Recall the notation $a/b\ominus c/d = (a-c)/(b-d)$. 
\begin{proof}
By Lemma~\ref{ctb} we have $\tb(L) = pq + tw(L,T) = pq - |q/p \bigcdot s|$.
  
We can assume that the characteristic foliation of $\partial B_i$ is in standard form for $0 \leq i \leq k$. Assume that the ruling curves have slope $q/p$ if that is not the dividing slope. Let $T_{i-1}$ and $T_i$ be boundary tori of $B_i$ with dividing slopes $s_{i-1}$ and $s_i$ respectively. If $v_i$ and $v_{i-1}$ are vectors in the integer lattice $\Z^2\cong H_1(B_i)\cong H^2(B_i, \partial B_i)$, then the Poincar\'e dual of the relative Euler class of $B_i$ is given by a curve in the class $\epsilon_i(v_i-v_{i-1})$, \cite[Proposition~4.22]{Honda00}. Now choose a leaf of slope $q/p$ in each boundary component and denote them by $L_{i-1}$ and $L_i$ respectively. Take a properly embedded annulus $A_i$ in $B_i$ where $\partial A_i = L_{i-1} \cup L_i$. Evaluating the relative Euler class on $A_i$ is given by computing the intersection of the curve representing $\epsilon_i(v_i-v_{i-1})$ and the curve of slope $q/p$. This is given by $\epsilon_i (s_{i} \ominus s_{i-1})\bigcdot q/p$. By Lemma~\ref{signs}, we have $s_i \bigcdot q/p \leq 0$ for $i \leq k$. It is also not hard to see that $|s_{i} \bigcdot q/p| < |s_{i-1} \bigcdot q/p|$. Thus $(s_{i} \ominus s_{i-1})\bigcdot q/p$ is always positive so $\epsilon_i (s_{i} \ominus s_{i-1})\bigcdot q/p = \epsilon_i |(s_{i} \ominus s_{i-1})\bigcdot q/p|$.

By Lemma~\ref{crot}, $\rot(L_0)$ is $p \rot(L')$. Since the difference between rotation numbers of $L_{i-1}$ and $L_i$ is the relative Euler class evaluated on $A_i$, $\rot(L_i) = \rot(L_{i-1}) - \epsilon_i |(s_{i} \ominus s_{i-1})\bigcdot q/p|$. Thus after summation, we obtain the desired formula.
\end{proof}

\begin{lemma}\label{destabtostd}
Suppose $K$ is a $q/p$-minimally thickenable knot. Then any Legendrian $L\in \mathcal{L}(K_{(p,q)})$ either destabilizes or is one of the standard $(p,q)$-cables of $K$. 
\end{lemma}
\begin{proof}
Given $L\in \mathcal{L}(K_{(p,q)})$, either $L$ is a Legendrian large cable, in which case it destabilizes to one of the $L_i^j$ for $i>n$ by Theorem~\ref{stabilizellc}, or by Lemma~\ref{negative-onboundary} we can find a solid torus $S$ with convex boundary having $2$ dividing curves such that $L$ is contained on the boundary of $S$. If $L$ does not intersect the dividing curves of $\partial S$ minimally, then there is a bigon cobounded by a segment of $L$ and the dividing set. From this we obtain a bypass for $L$ and we can destabilize $L$. Thus we can assume $L$ intersects the dividing curves minimally and we can arrange the foliation on $\partial S$ to be standard so that $L$ is a ruling curve or a Legendrian divide. In the latter case $L$ will be $L_n^j$ for some $j$. In the former case, assume the slope of the dividing curves on $\partial S$ is $s$. 

If $S$ is contained in a larger torus with slope $a_i$ for some $i < n$, then let $S'$ be such a torus with two dividing curves of slope $a_i$ where $a_i$ is the smallest $a_i$ larger than $s$. We can take a convex annulus with one boundary a ruling curve on $S'$ and the other $L$ on $S$. Clearly there cannot be bypasses on $A$ for $S'$ since attaching such a bypass would give a torus in $S'\setminus S$ of slope $a_{i-1}$ which would imply that $S'\setminus S$ is not minimally twisting. Thus we can use $A$ to destabilize $L$ to a ruling curve on $S'$. If $S'$ does not thicken to a solid torus of slope $a_{i+1}$ then it is one of the $N_i^j$ and we see that $L$ destabilized to $L_i^j$. If $S'$ does thicken to a solid torus $S''$ of slope $a_{i+1}$, then again take an annulus with boundary $q/p$ sloped ruling curves on $\partial S'$ and $\partial S''$, assume that this ruling curve is not $L$ on $\partial S'$. Now $A$ intersects the dividing set on $\partial S'$ more than the dividing set on $\partial S''$. Thus there is a bypass for $\partial S'$ that is disjoint from $L$. Attaching this bypass will produce a torus with dividing slope $a_{i+1}$ which will have to be contact isotopic to $\partial S''$. Thus we can take $L$ to be on $\partial S''$ and after further destabilizations it will be a ruling curve. We can iterate this process until we see $L$ destabilizing to $L_i^j$ for some $i<n$ and $j$, or $S$ is contained in a solid torus of slope $a_n$ and so we are left to consider the case when $i=n$. 
Now if $i=n$, then $S'$ is $N_n^j$ for some $j$ and one boundary of $A$ is a Legendrian divide of $S'$. We can again destabilize $L$ until it has $\tb(L)=pq-1$. Now the annulus will have only a single dividing curve with both ends on $L$. We cannot use the Realization Principle \cite[Section~3.3.1]{Honda00} to find another bypass for $L$. So now take a new annulus $A'$ that has one boundary of slope $q/p$ on $S'$ but intersecting the dividing curves twice, and the other boundary being $L$. Either the dividing curves run across $A'$ or there is a bypass for $L$ on $A'$. In the latter case we can use the bypass to destabilize $L$ to have $\tb(L)=pq$. Since  the destabilized $L$ is contained in $N_n^j$, it is $L_n^j$. (Recall, there is a unique Legendrian knot in $N_n^j$ with Thurston-Bennequin invariant $pq$.) Now if the dividing curves of $A'$ run across $A'$ then we can isotope $L$ to the Legendrian on $\partial N_n^j$ which can clearly destabilize to $L_n^j$.

We now assume that $S$ is not contained in any $N_i^j$ for $i\leq n$. In this case take the largest $a_i$ with $a_i<s$ and $i\leq n$. We know inside of $S$ there will be a convex torus with slope $a_i$. Let $N$ be a solid torus bounded by this torus. Now as above take an annulus $A$ between $\partial S$ and $\partial N$ with one boundary component being $L$ and the other being a $q/p$ ruling curve on $\partial N$ if $i<n$ and a Legendrian divide on $\partial N$ if $i=n$. If $i<n$ there can be no bypasses for $N$ along $A$ as attaching one would create a torus of slope $a_{i+1}$ which is contained in $S$, and this is not possible by the definition of $N$. Thus we can destabilize $L$ until it is a ruling curve on $N$. If $i=n$, notice that $N = N_n^j$ for some $j$. Now we can again destabilize $L$ until it has $\tb(L)=pq-1$. Since the annulus will have only a single dividing curve with both ends on $L$ as above, we cannot find another bypass for $L$ on $A$. So now take a new annulus $A'$ that has one boundary of slope $q/p$ on $\partial N_n^j$ but intersecting the dividing curves twice, and the other boundary being $L$. Either the dividing curves run across $A'$ and we can isotope $L$ to the Legendrian on $\partial N_n^j$ which can clearly destabilize to $L_n^j$, or there is a bypass for $L$ on $A'$ that we can use to destabilize $L$ to have $\tb(L)=pq$. Take an $I$-invariant neighborhood of $A'$ and glue this to $N_n^j$. Then we have a thickened solid torus with slope $a_n$ and more than two dividing curves. Now the destabilized $L$ is contained in this torus and since $K$ is $q/p$-minimally thickenable, we can further thicken the torus to have two dividing curves of slope $a_n$, which is contact isotopic to $\partial N_n^j$. Since $L$ is contained in $N_n^j$, it is $L_n^{j}$.
\end{proof}

\begin{remark}\label{remarktori}
  Before our next result we make an observation about tori in solid tori. Let $\xi$ be a tight contact structure on a solid torus $S$ with convex boundary of dividing slope $a_m$ and two dividing curves. Now let $T$ be a convex torus in $S$ that is parallel to the boundary and has dividing slope $a_l$ for some $l<m$ and two dividing curves. If $l\geq 0$ and $a_l$ is not on the interior of a continued fraction block with mixed signs, then there is a unique such torus in $S$. If $l\geq 0$ and $a_l$ is the $j^\text{th}$ term, for $j>1$, of a continued fraction block with $k_+>0$ positive basic slices and $k_->0$ negative basic slices, then there are $\min\{k_+,k_-, j-1\}+1$ different ways that $T$ can be realized. These come from shuffling the signs of the basic slices in the continued fraction block. The different tori have different numbers of positive basic slices coming before $T$. If $m\leq0$ then there are $l-m+1$ different tori $T$ coming from the boundary of stabilized knots. More precisely, a solid torus with convex boundary of slope $a_0=\lfloor q/p\rfloor$ corresponds to a Legendrian knot, and all the tori with integral slopes less than $a_0$ are the boundaries of neighborhoods of stabilizations of this knot. 
\end{remark}  

\begin{lemma}\label{negative-nondestab}
The standard cables $L_i^j$'s are not destabilizable for $i\not=n$. For $i=n$, $L_n^j$ is destabilizable if and only if $\partial N_n^j$ is contact isotopic to a torus in an $N_i^{j'}$ for some $i>n$.  
\end{lemma}
\begin{proof}
For $i > n$, the lemma follows from Theorem~\ref{llc-isotopy}. The $i=n$ case follows from Theorem~\ref{stabilizellc} since it shows exactly which $L_n^j$'s are stabilizations of Legendrian large cables. 

Now suppose $L_i^j$ destabilizes to $L$ for $i < n$ with $\tb(L)=\tb(L_i^j)+1$. Then there is a convex torus $T$ bounding a solid torus representing $K$ on which $L$ sits. We can arrange the foliation of $T$ so that $L_i^j$ also sits on $T$ along with some bypasses that can be used to destabilize it to $L$. Thus we can assume that $\partial N_i^j$ agrees with $T$ along $L_i^j$. By Lemma~\ref{smooth-isotopy}, there is a smooth isotopy from $\partial N_i^j$ to $T$ that fixes $L_i^j$ and thus by Colin's isotopy discretization~\cite{Honda02} we can get from $T_0=\partial N_i^j$ to $T_m=T$ by a sequence of tori $T_l$ such that $T_{l+1}$ is obtained from $T_l$ by a bypass attachment from the inside or the outside of $T_l$ that is disjoint from $L_i^j$ (we say a bypass is attached from the outside if it is attached outside the solid torus that $T_l$ bounds and attached to the inside otherwise). Denote by $s_l$ the slope of the dividing curves on $T_l$ and by $S_l$ the solid torus bounded by $T_l$. Now we inductively show that each $T_l$ satisfies the following properties, which proves that $L$ cannot exist. 
\begin{itemize}
  \item $a_i\leq s_l < a_{i+1}$, and
  \item $T_l$ is the front boundary of $B_l=T^2\times [0,1]$ such that the back boundary is $\partial N_i^j$, and
  \item $L$ cannot sit on $T_l$.
\end{itemize}
Here, we can think of $B_l$ as a neighborhood of the boundary of a solid torus $S_l$. Clearly these are true for $T_0$. Now assume they are true for $T_l$ and we then show that they are also true for $T_{l+1}$. 

First, if $s_{l+1} \geq a_{i+1}$, this implies that the bypass is attached from outside of $S_l$. However, since $\partial N_i^j$ is contained in $S_l$ by hypothesis, this would mean that $N_i^j$ is contained in a solid torus with dividing slope $a_{i+1}$, contrary to the definition of $N_i^j$. 

If $s_{l+1}<a_i$, this implies that $s_l=a_i$ and the bypass is attached from inside of $S_l$. Then the dividing curves on $T_{l+1}$ would intersect a $q/p$ slope curve more times than the dividing curves on $T_l$ by Remark~\ref{genintersect}. However, $L_i^j$ intersects the dividing curves on $T_l$ minimally and the bypass was attached in the complement of $L_i^j$, so the number of intersections with the dividing curves cannot increase.

The last case is when $a_i \leq s_{l+1} < a_{i+1}$. If the bypass is attached from outside of $S_l$, clearly $B_{l+1}$ is  $B_l$ with a bypass attached to its front face. If the bypass is attached from inside, then $S_{l+1}$ is contained in $S_l$. By Remark~\ref{remarktori}, there is a unique torus with dividing slope $a_i$ in $S_l$ so $S_{l+1}$ also contains a torus, which is contact isotopic to $\partial N_i^j$ by hypothesis. Thus we can take $B_{l+1}$ with the back boundary $\partial N_i^j$ and the front boundary $T_{l+1}$. Now by Remark~\ref{genintersect}, we have
\[
  |a_i\bigcdot q/p|\leq|s_l\bigcdot q/p|. 
\]
Since $L$ is a destabilization of $L_i^j$, $L$ should have bigger twisting number with respect to $T_{l+1}$ compared to $L_i^j$.
\[
  tw(L,T_{l+1}) > tw(L_i^j,T_{l+1})=-|a_i\bigcdot q/p|.
\]
However, any Legendrian knot in $\mathcal{L}(K_{(p,q)})$ that sits on $T_{l+1}$ has at most $-|s_{l+1}\bigcdot q/p|$ twisting number with respect to $T_{l+1}$. Thus $L$ cannot sit on $T_{l+1}$. This completes our induction and the proof.
\end{proof}

  \begin{lemma}\label{negative-nonisotopic}
  If $K$ is $q/p$-minimally thickenable, then $L_i^j$ and $L_{i'}^{j'}$ are Legendrian isotopic if and only if $i = i'$ and $j = j'$.
\end{lemma}
\begin{proof}
Clearly $i=i'$ and $j=j'$ imply $L_i^j$ and $L_{i'}^{j'}$ are Legendrian isotopic. For the converse we notice that by Lemma~\ref{negative-tbrot}, $L_i^j$ and $L_{i'}^{j'}$ have different Thurston-Bennequin invariants for $i\not = i'$. Thus we know that $i=i'$ if $L_i^j$ is isotopic to $L_{i'}^{j'}$. We now show $j=j'$.  

We first deal with the $i=n$ case. Assume that $L_n^j$ and $L_n^{j'}$ are Legendrian divides of $N_n^j$ and $N_n^{j'}$ respectively and are Legendrian isotopic. Denote $L_n^j=L_{n}^{j'}$ by $L$. After contact isotopy, we can assume that $\partial N_n^j$ and $\partial N_n^{j'}$ share a common Legendrian divide $L$.  By Lemma~\ref{smooth-isotopy}, there is a smooth isotopy from $\partial N_n^j$ to $\partial N_n^{j'}$ fixing $L$. By Colin's isotopy discretization~\cite{Honda02} we can get from $T_0=\partial N_n^j$ to $T_m=\partial N_n^{j'}$ by a sequence of tori $T_l$ such that $T_{l+1}$ is obtained from $T_l$ by a bypass attachment from the inside or the outside of $T_l$ that is disjoint from $L$ (we say a bypass is attached from the outside if it is attached outside the solid torus that $T_l$ bounds and attached to the inside otherwise). We inductively claim that $T_l$ is contained in an $I$-invariant neighborhood of a torus that is contact isotopic to $\partial N_n^j$. Given this then clearly $\partial N_n^{j'}$ is contact isotopic to $\partial N_n^j$, contradicting their definition, so $L_n^j$ is not Legendrian isotopic to $L_{n}^{j'}$. 
  
To verify the claim we note that $T_0=\partial N_n^j$ is contained in an $I$-invariant neighborhood of $\partial N_n^j$. Now inductively assume that $T_l$ is contained in such an $I$-invariant neighborhood $B$ of a torus that is contact isotopic to $\partial N_n^j$. Notice that this implies that both boundary components of $B$ are isotopic to $\partial N_n^j$. The torus $T_l$ divides $B$ into two pieces, $B_+$ containing the front face of $B$ and $B_-$ containing the back face. Suppose $T_{l+1}$ is obtained by attaching a bypass to the outside of $T_l$. In this case, notice that the slope of the dividing curves does not change since $L$ will still have twisting zero with respect to the torus. Thus the bypass is trivial or increases or decreases the number of dividing curves. Since $K$ is $q/p$-minimally thickenable, there is a convex torus $T'$ that is outside of the solid torus that $T_{l+1}$ bounds that has two dividing curves of slope $q/p$. Now $T_l$ and $T'$ cobound a thickened torus $B'$ that contains $T_{l+1}$ and the contact structure on the union $B_-\cup B'$ is an $I$-invariant neighborhood of a torus isotopic to $\partial N_n^j$, as desired. A similar, but easier, argument works if the bypass is attached from the inside since the contact structure on the solid torus bounded by $T_l$ contains convex tori of different slopes and hence when attaching a bypass if the number of dividing curves changes, one can find a convex torus with just two dividing curves further inside the solid torus.   

Now we deal with the $i < n$ case. Suppose $L_i^j$ and $L_{i}^{j'}$ are ruling curves on $N_i^j$ and $N_i^{j'}$, and Legendrian isotopic. After contact isotopy, $N_i^j$ and $N_{i}^{j'}$ share a common Legendrian ruling curve $L=L_i^j=L_i^{j'}$. By Lemma~\ref{smooth-isotopy}, there is a smooth isotopy from $\partial N_i^j$ to $\partial N_i^{j'}$ that does not move $L$. By Colin's isotopy discretization, we can get from $T_0=\partial N_i^j$ to $T_m=\partial N_i^{j'}$ by a sequence of tori $T_l$ such that $T_{l+1}$ is obtained from $T_l$ by a bypass attachment from the inside or the outside of $T_l$ that is disjoint from $L$. In the proof of Lemma~\ref{negative-nondestab}, we showed that there is $B_l=T^2\times [0,1]$ such that the front boundary is $T_l$ and the back boundary is $\partial N_i^j$ for all $l=0,...,m$. Then clearly $B_m$ is an $I$-invariant neighborhood from $\partial N_i^j$ to $\partial N_i^{j'}$, contradicting the fact that $N_i^j$ and $N_i^{j'}$ are not contact isotopic. 

For $i > n$, the lemma follows from Theorem~\ref{llc-isotopy}.
\end{proof}

\begin{lemma}\label{relation2}
If $L = S_+^{k}S_-^{l}(L_i^j) = S_+^{k'}S_-^{l'}(L_{i'}^{j'})$ for $i\leq n$ and $i'\leq n$, then they are related as indicated in Item~\eqref{neg7} of Theorem~\ref{classification-negative}.
\end{lemma}

\begin{proof}
Suppose $L=S_+^{k}S_-^{l}(L_i^j) = S_+^{k'}S_-^{l'}(L_{i'}^{j'})$. We can put $L$ on $\partial N_i^j$ and on $\partial N_{i'}^{j'}$ and thus after contact isotopy we can assume that $\partial N_i^j$ and $\partial N_{i'}^{j'}$ agree along $L$. Now using Lemma~\ref{smooth-isotopy} we can find a smooth isotopy from $\partial N_i^j$ and $\partial N_{i'}^{j'}$ that fixes $L$. By Colin's isotopy discretization~\cite{Honda02} we can get from $T_0=\partial N_i^j$ to $T_m=\partial N_{i'}^{j'}$ by a sequence of tori $T_u$ such that $T_{u+1}$ is obtained from $T_u$ by a bypass attachment from the inside or the outside of $T_u$ that is disjoint from $L$. Denote by $s_u$ the slope of the dividing curves on $T_u$ and by $S_u$ the solid torus bounded by $T_u$.

We begin by claiming that we can alter this sequence of tori so that we never have $q/p$ strictly between $s_u$ and $s_{u+1}$. To see this suppose that $s_u<q/p<s_{u+1}$ (the other case being analogous). So $T_u$ and $T_{u+1}$ cobound $B=T^2\times I$ and inside $B$ there is a convex torus $T$ with dividing slope $q/p$ and two dividing curves. Notice that $B\setminus T$ is two thickened tori $B_-$ and $B_+$ where $B_-$ contains $T_u$ and $B_+$ contains $T_{u+1}$. Now inside of $B_-$ we can take a copy $T'$ of $T$ on which $L$ sits (there is clearly such a torus as $L$ destabilizes to a Legendrian divide on $T$). Now inside of $B_-$ we can isotope $T'$ so that it agrees with $T_u$ along $L$. Thus we can use isotopy discretization to find a sequence of tori going from $T_u$ to $T'$ that are each related to the next by a bypass attachment. Since all the tori are contained in $B_-$ we know the slopes of the tori are all between the slope of $s_u$ and $q/p$. In particular, they are all less than or equal to $q/p$. Now similarly in $B_+$ we can take a copy $T''$ of $T$ that contains $L$ and as above we can find a sequence of tori from $T''$ to $T_{u+1}$ that are related by bypass attachments and have slopes between $q/p$ and $s_{u+1}$. Since $T'$ and $T''$ can be taken to be contact isotopic, we have proven our claim. 

We will now show that 
\begin{enumerate}
  \item If $s_u\geq q/p$ and $s_{u+1} \geq q/p$, for any $N_n^{j_u}\subseteq S_u$ and $N_n^{j_{u+1}}\subseteq S_{u+1}$ there is a super commensurating torus which contains $L$, and 
  \item If $s_u \leq q/p$ for $u=a,a+1,...,b$, for $T_a$ and $T_b$ there is a commensurating torus with slope $a_s$ for some $s\in\mathbb{Z}$ which contains $L$. 
\end{enumerate}
Once we have proven this, take a subsequence ${T_{k_0},...,T_{k_{r+1}}}$ such that $T_{k_0} = T_0$, $T_{k_{r+1}} = T_m$ and $T_{k_1},...,T_{k_r}$ have the slope $q/p$. Let $L_{k_u}$ be the Legendrian divide of $T_{k_u}$ for $u=1,...,r$. Then we have a sequence $L_i^j = L_{k_0}, L_{k_1}, \ldots, L_{k_r}, L_{k_{r+1}} = L_{i'}^{j'}$ such that $L_{k_u}$ and $L_{k_{u+1}}$ for each $u=0,...,r$ have a common stabilization that is related by (1) or (2); in the case of (1) the Legendrian ruling curve of the super commensurating torus is a common stabilization of $L_{k_u}$ and $L_{k_{u+1}}$, and in the case of (2) the Legendrian ruling curve of the commensurating torus is a common stabilization of $L_{k_u}$ and $L_{k_{u+1}}$. After that, one may take a solid torus $S$ with the minimal slope inside of the super commensurating torus that contains $S_{k_u}\cup S_{k_{u+1}}$ (after contact isotopy). This will be the minimal super commensurating torus for $T_{k_u}$ and $T_{k_{u+1}}$. Similarly take a solid torus with the maximal slope $a_s$ containing the commensurating torus and contained in $S_{k_u}\cap S_{k_{u+1}}$ (after contact isotopy). This will the maximal commensurating torus for $T_{k_u}$ and $T_{k_{u+1}}$. Thus $S_+^{k}S_-^{l}(L_i^j)$ and $S_+^{k'}S_-^{l'}(L_{i'}^{j'})$ will be related by a sequence of minimal super commensurating and maximal commensurating tori (and further stabilizations) as claimed in Item~\eqref{neg7} of Theorem~\ref{classification-negative} and the lemma is complete.

To verify (1) notice that each $S_u$ contains some $N_n^{j_u}$ for some $j_u$ if $s_u\geq q/p$. Now $N_{n}^{j_u}$ and $N_{n}^{j_{u+1}}$ are both contained in either $S_u$ or $S_{u+1}$, depending on whether the bypass from $T_u$ to $T_{u+1}$ was attached from the outside or inside. Also $L$ sits on both $T_u$ and $T_{u+1}$ so there is a super commensurating torus for $N_n^{j_u}$ and $N_n^{j_{u+1}}$ containing $L$.

To verify (2), we first show for each $u$ there is a convex torus $T^c_u$ with dividing slope $a_{k_u}$ for some $k_u \in \mathbb{Z}$ that contains $L$ and is contained in $S_u\cap S_{u+1}$ (after contact isotopy). To this end, notice that $S_u$ is contained in or contains $S_{u+1}$, depending on whether the bypass from $T_u$ to $T_{u+1}$ was attached on the inside or outside. Suppose $S_u\subseteq S_{u+1}$ then $T^c_u$ is simply the convex torus inside $S_u$ with slope $a_{k_u}$ where $k_u$ is the largest integer such that $a_{k_u}$ is less than or equal to $s_u$. Notice that $a_{k_u}$ is less than or equal to $a_n=q/p$ by our assumption on the slopes of the $T_u$. Clearly $L$ sits on $T^c_u$ as in the proof of Lemma~\ref{destabtostd}. The same argument gives $T_u^c$ when $S_{u+1}\subseteq S_u$.

We now show that there is a commensurating torus for $T_a$ and $T_b$ containing $L$. We will do this by inductively showing that there is a convex torus with slope $a_s$ for some $s$ that contains $L$ and is contained in $S_a\cap S_{u}$ (after contact isotopy). The base case is clear as we have already proven there is a commensurating torus for $S_a$ and $S_{a+1}$ containing $L$. We now assume we have a commensurating torus $T$ for $S_a$ and $S_{u}$ containing $L$. There are two cases to consider. The first is when $s_{u}\in (a_s, a_{s+1})$ for some $s$. In this case $s_{u+1}$ will be in $[a_s,a_{s+1}]$ since $T_u$ and $T_{u+1}$ are related by a single bypass attachment. Thus we see that the slope of $T$ must be less than or equal to $a_s$ and we know that the slope of $T^c_{u}$ is $a_s$. Using the same arguments as in Remark~\ref{remarktori} there is a unique torus in $S_{u}$ with slope $a_s$ and all tori in $S_{u}$ with slope smaller than $a_s$ are contained in this torus. Thus $T$ is contained in a solid torus bounded by $T_u^c$ and we see that $T$ is contained in $S_a\cap S_{u+1}$, finishing the inductive step in this case. 

The other case is when $s_{u}=a_s$ for some $s$. Since $T_{u+1}$ is obtained from $T_{u}$ by a single bypass attachment, we know $s_{u+1}$ is in $[a_{s-1},a_{s+1}]$ or larger than $a_n=q/p$. By hypothesis, it cannot be larger than $a_n$, so we only need to consider the former case. If $s_{u+1} \in [a_{s}, a_{s+1}]$, then $S_{u+1}$ contains $S_{u}$ and hence also $T$. Thus $T$ is contained in $S_a\cap S_{u+1}$. If $s_{u+1}\in (a_{s-1}, a_s)$, then as in Remark~\ref{remarktori}, there is a unique solid torus in $S_{u}$ with boundary slope $s_{u+1}$ and all tori with slope less than or equal to $s_{u+1}$ are contained in the solid torus. Clearly this solid torus is $S_{u+1}$. If the slope of $T$ is less than or equal to $a_s$, it must also be contained in $S_{u+1}$ and hence $T\subset S_a\cap S_{u+1}$. 

The last case is when $s_{u+1}=a_{s-1}$. If the slope of $T$ is $a_s$, this implies that $T = T_{u}$ so $T_{u}^c\subset S_a\cap S_{u+1}$. Now suppose the slope of $T$ is less than or equal to $a_{s-1}$. If $T$ is contained in $S_{u+1}$, then we obviously have $T\subset S_a\cap S_{u+1}$. If $T$ is not contained in $S_{u+1}$, then by Remark~\ref{remarktori}, $s_u$ and $s_{u+1}$ and $s_T$, the slope of $T$, are vertices in a continued fraction block (since if not there would be a unique torus of slope $s_{u+1}$ in $S_u$ and it would have to contain $T$). Notice that there is another solid torus $S'_{u+1}$ with boundary slope $s_{u+1}$ that contains $T$ and is contained in $S_u$. Suppose the sign of the basic slice with the front boundary $\partial S_u$ and the back boundary $\partial S_{u+1}$ is positive (the other case being analogous). Then the sign of the basic slice with the front boundary $\partial S_u$ and the back boundary $\partial S'_{u+1}$ should be negative (or else $S'_{u+1}$ is contact isotopic to $S_{u+1}$ and we are done). This implies that all basic slices between $\partial S_{u}$ and $T$ have negative signs since if it is not true, then $T$ is contained in $S_{u+1}$ by shuffling the signs of basic slices. Let $L_u$ be the ruling curve on $\partial S_{u}$. Since $L$ sits on $\partial S_{u+1}$, $L$ is a further stabilization of $S_+^{x}(L_u)$ where $x=|(s_u\ominus s_{u+1})\bigcdot q/p|$ by Lemma~\ref{negative-tbrot}. Similarly, since $L$ sits on $T$, $L$ is a further stabilization of $S_-^y(L_u)$ where $y=|(s_u\ominus s_T)\bigcdot q/p|$. Notice that there is another convex torus $T'$ whose slope is a vertex of the continued fraction block above which is contained in a solid torus bounded by $T$ related by a single bypass attachment such that the sign of the basic slice with front boundary $T$ and the back boundary $T'$ is positive. Since the ruling curve on $T'$ is Legendrian isotopic to $S_+^xS_-^y(L_u)$, we may further stabilize it to obtain $L$ and put it on $T'$. Clearly $T' \subset S_a\cap S_{u+1}$.
\end{proof}

\begin{proof}[Proof of Theorem~\ref{classification-negative}.]
Item~\eqref{neg1} of the theorem follows from Lemmas~\ref{destabtostd} and~\ref{negative-nondestab}. Item~\eqref{neg2} is precisely the content of Lemma~\ref{negative-nonisotopic}. Items~\eqref{neg3} and~\eqref{neg4} follow from Lemma~\ref{negative-tbrot} if $i\leq n$ and for $i>n$ from Theorem~\ref{llc-model} and Lemma~\ref{llcrot}. Items~\eqref{neg8} and~\eqref{neg9} follow from Theorems~\ref{stabilizellc} and~\ref{llc-isotopy}, respectively. We also observe that Item~\eqref{neg10} is an immediate consequence of the other items. 

So we are left to check Items~\eqref{neg5}, \eqref{neg6}, and~\eqref{neg7}. We begin with Item~\eqref{neg5}. Suppose $N$ is a maximal commensurating torus for $N_i^j$ and $N_{i'}^{j'}$ (see Section~\ref{negcablesec} for the terminology) with dividing slope $a_m$ for $m\leq\min\{i, i'\}$, so after a contact isotopy we can assume that $N$ is a subset of $N_i^j\cap N_{i'}^{j'}$. Let $L$ be a ruling curve on $\partial N$ of slope $q/p$. Take an annulus $A$ in $N_i^j\setminus N$ with one boundary component on $L$ and the other on $L_i^j$ and make it convex. Notice that there cannot be a bypass for $\partial N_i^j$ along $A$, since if there were then attaching it would result in a torus in $N_i^j\setminus N$ with dividing slope $a_{i-1}$. Since this bypass can be attached in the complement of (a copy of) $L_i^j$ and slope $a_{i-1}$ curves intersect a $q/p$ curve more than slope $a_i$ curves, this cannot happen. So all the dividing curves on $A$ that start on $L_i^j$ must go from one boundary component to the other. Since the dividing curves on $T$ intersect $L$, $2|(a_i\ominus a_m) \bigcdot q/p|$ more times than the ones on $\partial N_i^j$ intersect $L_i^j$, we see that there are $|(a_i\ominus a_m)\bigcdot q/p|$ bypasses for $L$ along $A$. They will destabilize $L$ to $L_i^j$. We are left to check the signs of the stabilizations. To this end decompose $N_i^j \setminus N$ into basic slices $B_{m+1}, \ldots ,B_{i}$ where $B_s$ is a basic slice with dividing slopes $a_{s-1}$ and $a_{s}$. Suppose $B_{m+1}$ is a positive basic slice. Let $A'$ be an annulus in $B_{m+1}$ with one boundary component on $L$ and the other on a $q/p$ ruling curve on the other boundary component of $B_{m+1}$. We see as above that the number of bypasses for $L$ along $A'$ is $|(a_{m+1}\ominus a_{m}) \bigcdot q/p|$. And they will all be positive. So $L$ positively destabilizes $|(a_{m+1} \ominus a_m) \bigcdot q/p|$ times to a ruling curve on the positive boundary of $B_{m+1}$. If $B_{m+1}$ was a negative basic slice then these would be negative destabilizations. Continuing through the other basic slices we arrive at the fact that $L=S^k_+S^l_-(L_i^j)$ where $k$ and $l$ are determined as in the statement of the theorem. We can similarly check that $L$ destabilizes to $L_{i'}^{j'}$ as claimed. 

Item~\eqref{neg6} is addressed similarly to Item~\eqref{neg5}.

Item~\eqref{neg7} is exactly the content of Lemma~\ref{relation2}.
\end{proof}

\begin{proof}[Proof of Theorem~\ref{classification-in-between}.]
This is almost identical to the proof of Theorem~\ref{classification-negative}. We leave it as an exercise for the reader.
\end{proof}

\section{General results on cables}\label{grc}
We now turn to the proof of Theorem~\ref{tm1} that says a knot type $K$ is Legendrian simple if and only if its $(p,q)$-cables with $q/p\geq \lceil \omega(K) \rceil$ are.
\begin{proof}[Proof of Theorem~\ref{tm1}]
This is an immediate consequence of Theorem~\ref{classification-positive}.
\end{proof}

Recall that Theorem~\ref{tm2} says that a uniform thick knot type is transversely simple if and only if sufficiently negative cables are transversely simple. 
\begin{proof}[Proof of Theorem~\ref{tm2}]
Suppose $K$ is a uniform thick knot type. Since $K$ is uniformly thick there is a finite number of Legendrian knots $L_1,\ldots, L_n$ with the maximal Thurston-Bennequin invariant and all other Legendrian knots are stabilizations of one of these. 

Suppose that $K$ is transversely simple. Then by Theorem~\ref{traniso} there is some maximal $m$ such that the Legendrian knots $S_-^m(L_1),\ldots, S_-^m(L_n)$ are distinguished by the rotations numbers. Suppose we have numbered $L_i$ so that $S_-^m(L_1),\ldots, S_-^m(L_k)$, for some $k\leq n$, are distinct Legendrian knots and the others are isotopic to one of these. Now there is another $m'$ such that each of the $S_-^m(L_i)$ has an $m'$-fold stabilization (possibly involving both positive and negative stabilizations) that is isotopic to appropriate $m'$-fold stabilization of $S_-^m(L_j)$ for all other $j$. 

Let $l=\tbb(K)-m-m'$. We claim that any $(p,q)$-cable of $K$ with $q/p<l$ is transversely simple. We notice that if one considers all the Legendrian knots in $\mathcal{L}(K)$ with Thurston-Bennequin invariant equal to $l$, they are determined by their rotation numbers. Suppose there are $n'$ of these, so there are exactly $n'$ solid tori in the knot type $K$ with convex boundary having two dividing curves and dividing slope $l$ and they are distinguished by the relative Euler class of their complements (that is the rotation number of the associated Legendrian knot). Moreover, all solid tori with convex boundary and dividing slope less than $l$ can be enlarged to one of these solid tori. Now Theorem~\ref{classification-negative} says that a $(p,q)$-cable of $K$ is Legendrian simple if $q/p\leq l$. Thus such cables are also transversely simple. 

Now assume that $K$ is not transversely simple. Then, there exist two Legendrian knots $L$ and $L'$ in the knot type $K$ that have the same classical invariants, but $S_-^m(L)$ and $S_-^m(L')$ are distinct for all $m$. Notice that this implies that for all $s<\tb(L)$ there are two distinct solid tori with convex boundary having two dividing curves and slope $s$ whose complements have the same relative Euler number but are not contact isotopic (these come form neighborhoods of  $S_-^m(L)$ and $S_-^m(L')$  when $m$ is an integer and are contained in such neighborhoods otherwise). Now Theorem~\ref{classification-negative} says that for any $q/p<\tb(L)$ there will be standard Legendrian $(p,q)$-cables of $K$ that are distinct and remain distinct after arbitrarily many negative stabilizations. 
\end{proof}

Finally we prove Theorem~\ref{tm3} which says that if $K$ is uniformly thick and transversely simple then sufficiently negative cables of $K$ will be Legendrian simple (even if $K$ itself is not). 
\begin{proof}[Proof of Theorem~\ref{tm3}]
If $K$ is partially uniformly thick and transversely simple, then let $n$ be the integer to which all solid tori can thicken. Considering the proof of Theorem~\ref{tm2}, let $l$ be as in the third paragraph. Then the argument in the proof of Theorem~\ref{tm2} shows that any $(p,q)$-cable of $K$ with $q/p<\min \{n,l\}$ is transversely simple. 
\end{proof}

\def\cprime{$'$}


\end{document}